\documentclass[a4paper,10pt]{article}
\usepackage[utf8x]{inputenc}
\usepackage{tracefnt,amsmath,tabu,array}
\usepackage{amssymb,graphicx,setspace,amsfonts,amsbsy}
\usepackage{pifont,latexsym,ifthen,amsthm,rotating,calc,textcase,booktabs,cancel,slashed}

\newtheorem{theorem}{Theorem}[section]
\newtheorem{lemma}[theorem]{Lemma}
\newtheorem{corollary}[theorem]{Corollary}
\newtheorem{proposition}[theorem]{Proposition}
\newtheorem{definition}[theorem]{Definition}

\newtheorem{remark}[theorem]{Remark}
\newcommand{\filledbox}{\leavevmode
  \hbox to.77778em{%
  \hfil\vbox to.675em{\hrule width.6em height.6em}\hfil}}

\newcommand{\Rm}{{\mathbb R}}

\newcommand{\eps}{\varepsilon}

\begin{document}
\tabulinesep=1.0mm
\title{Long time behaviour of finite-energy radial solutions to energy subcritical wave equation in higher dimensions\footnote{MSC classes: 35L05, 35L71.}}

\author{Ruipeng Shen\\
Centre for Applied Mathematics\\
Tianjin University\\
Tianjin, China}

\maketitle

\begin{abstract}
  We consider the defocusing, energy subcritical wave equation $\partial_t^2 u - \Delta u = -|u|^{p-1} u$ in 4 to 6 dimensional spaces with radial initial data. We define $w=r^{(d-1)/2} u$, reduce the equation above to one-dimensional equation of $w$ and apply method of characteristic lines. This gives scattering of solutions outside any given light cone as long as the energy is finite. The scattering in the whole space can also be proved if we assume the energy decays at a certain rate as $x\rightarrow +\infty$. This generalize the 3-dimensional results in Shen \cite{shenenergy} to higher dimensions.
 \end{abstract}

\section{Introduction}

\subsection{Background}

In this work we consider defocusing wave equation in dimensions $3\leq d \leq 6$.
\[
 \left\{\begin{array}{ll} \partial_t^2 u - \Delta u = - |u|^{p-1}u, & (x,t) \in \Rm^d \times \Rm; \\
 u(\cdot, 0) = u_0; & \\
 u_t (\cdot,0) = u_1. & \end{array}\right.\quad (CP1)
\]
\paragraph{Critical Sobolev spaces} The space $\dot{H}^{s_p} \times \dot{H}^{s_p-1}$ with $s_p = d/2-2/(p-1)$ is known as the critical Sobolev space of (CP1). This is because the $\dot{H}^{s_p} \times \dot{H}^{s_p-1}$ norm of initial data is preserved if we apply the natural rescaling transformation $(\mathbf{P}_\lambda u) (x,t) = \lambda^{-2/(p-1)} u(x/\lambda, t/\lambda)$. Given any constant $\lambda \in \Rm^+$, $\mathbf{P}_\lambda$ is an element in the symmetric group of (CP1), i.e. $\mathbf{P}_\lambda u$ is a solution to (CP1) as long as $u$ is. In particular, the case with $p=p_e(d)\doteq 1+4/(d-2)$ and $s_p=1$ is called the energy critical case; the case with $p=p_c(d) = 1+4/(d-1)$ and $s_p=1/2$ is called the conformal case. In this work we consider the energy subcritical, superconformal case with $1+4/(d-1)\leq p<1+4/(d-2)$.
\paragraph{Local theory} We may obtain the existence and uniqueness of local solutions by combining suitable Strichartz estimates with a fixed-point argument. More details about this kind of argument can be found in Kapitanski \cite{loc1} and Lindblad-Sogge \cite{ls}, for example. Suitable solutions also satisfy the energy conservation law
\[
 E(u,u_t) = \int_{\Rm^d}\left(\frac{1}{2}|\nabla u(x,t)|^2 + \frac{1}{2}|u_t(x,t)|^2 + \frac{1}{p+1}|u(x,t)|^{p+1}\right) dx = \hbox{Const}.
\]
\paragraph{Scattering} The global behaviour of solutions to defocusing wave equations is less complicated than those in the focusing case. It is conjectured that all solutions $u$ to (CP1) with initial data in the critical Sobolev space scatter in both two time directions. In other words, when $t \rightarrow \pm \infty$ a solution to (CP1) always becomes more and more like a free wave, i.e. a solution to the homogenous linear wave equation $\partial_t^2 u - \Delta u = 0$. In 1990's the energy critical case of this conjecture was proved by Grillakis \cite{mg1} in dimension 3 and Grillakis \cite{mg2}, Shatah-Struwe \cite{ss1,ss2} in higher dimensions. The energy supercritical case $p>p_e(d)$ and energy subcritical case $p<p_e(d)$ seem to be more difficult. Whether this conjecture is true or not in these situations remains to be an open problem, as far as the author knows, although there are many works proving the scattering of solutions with different kinds of additional assumptions on initial data or, sometimes, global behaviour of solutions. Some examples of these works are given below.
\paragraph{Scattering of bounded solutions} If the solution $u$ is known to be uniformly bounded in the critical Sobolev space for all time in its maximal lifespan, then we may apply the compactness-rigidity argument introduced in Keng-Merle \cite{kenig, kenig1} to prove the scattering of $u$. There are many works of this kind for different ranges of $d$ and $p$, sometimes with a radial assumption. Please see Duyckaerts et al. \cite{dkm2}, Kenig-Merle \cite{km}, Killip-Visan \cite{kv2} (dimension 3), Killip-Visan \cite{kv3} (all dimensions) for energy supercritical case and Dodson-Lawrie \cite{cubic3dwave},  Dodson et al. \cite{nonradial3p5}, Shen \cite{shen2} (dimension 3),  Rodriguez \cite{sub45} (dimension 4 and 5) for energy subcritical case. Please note that the results of this kind hold in both defocusing and focusing cases, except for the focusing energy critical equation. This is because a soliton is available in the critical Sobolev space for the focusing equation in the energy critical case but not in supercritical or subcritical cases. 
\paragraph{Better initial data} Scattering has also been proved with stronger assumptions on the initial data. Dodson \cite{claim1} gives a proof of the conjecture above for cubic 3D wave equation with radial data. In the non-radial case there are a lot of results assuming the energy of initial data to decay at certain rate, i.e. 
 \[
  E_\kappa(u_0,u_1)\doteq \int_{\Rm^d} (1+|x|)^\kappa \left(\frac{1}{2}|\nabla u_0(x)|^2 + \frac{1}{2}|u_1(x)|^2 + \frac{1}{p+1}|u_0(x)|^{p+1}\right) < +\infty.
 \]
For example, the conformal conservation law (see Ginibre-Velo \cite{conformal2} and Hidano \cite{conformal}) leads to the scattering of solutions for all $d\geq 3$ and $1+4/(d-1)\leq p<1+4/(d-2)$, if the initial data $(u_0,u_1)$ satisfy roughly $E_2(u_0,u_1) < + \infty$. Yang \cite{yang1} considers the energy momentum tensor and its associated current, then gives a scattering result with a weaker assumption on the initial data $E_\kappa(u_0,u_1)<+\infty$, as long as $p$ and $\kappa$ satisfy
\begin{align*}
 &\frac{1+\sqrt{d^2+4d-4}}{d-1} < p < p_e(d),& &\kappa> \max\left\{\frac{4}{p-1}-d+2,1\right\};&
\end{align*}
Recently in works \cite{shen3dnonradial, shenhd} the author introduces the inward/outward energy theory for non-radial solutions and proves the scattering result for initial data satisfying $E_\kappa(u_0,u_1)<+\infty$ with 
\[
  \kappa >  \kappa_0(d,p) = \frac{(d+2)(d+3)-(d+3)(d-2)p}{(d-1)(d+3)-(d+1)(d-3)p} \in (0,1), \quad 1+4/(d-1)< p<1+4/(d-2).
\]
\paragraph{Method of characteristic lines} All the results above have a thing in common: the initial data are assumed to be in the critical Sobolev space of (CP1). Although we do not assume this explicitly in some results above, this is actually a direct consequence of $E_\kappa(u_0,u_1)<+\infty$ by Sobolev embedding, as long as $\kappa$ is large enough. By contrast, the author in his recent paper \cite{shenenergy} proves the scattering of radial solutions to defocusing, energy subcritical 3D wave equation with a much weaker assumption on the decay rate of initial data 
\[
 E_\kappa(u_0,u_1)< +\infty, \quad \kappa > \frac{5-p}{p+1}, \quad p \in [3,5).
\]
The decay rate is so low that the initial data, thus data at any time are not necessarily contained in the space $\dot{H}^{s_p} \times \dot{H}^{s_p-1}$. As a result we use the energy space to describe the scattering instead. Namely by scattering (in the positive time direction, for example) we mean that there exists a finite-energy free wave $\tilde{u}$ so that 
\[
 \lim_{t\rightarrow +\infty} \left\|(u(\cdot,t),u_t(\cdot,t)) - (\tilde{u}(\cdot,t),\tilde{u}_t(\cdot,t))\right\|_{\dot{H}^1 \times L^2(\Rm^3)} = 0.
\]
If we only assume the finiteness of $E(u_0,u_1)$ instead of $E_\kappa(u_0,u_1)$, we can still obtain the scattering outside any given light cone, i.e. 
\[
 \lim_{t\rightarrow +\infty} \left\|(\nabla u(\cdot,t),u_t(\cdot,t)) - (\nabla \tilde{u}(\cdot,t),\tilde{u}_t(\cdot,t))\right\|_{L^2(\{x: |x|>t-\eta\})} = 0, \quad \forall \eta\in \Rm.
\]
This kind of scattering phenomena have not been discovered in previous works, as far as the author knows. To prove these results we first reduce the radial case of 3D wave equation to a one-dimensional wave equation and then apply the method of characteristic lines. More precisely, given a radial solution $u$ to (CP1) in 3-dimensional case, we may define\footnote{Given any radial function $u$, $u(r,t)$ represents the value $u(x,t)$ with $|x|=r$.} $w(r,t) = r u(r,t)$ and rewrite the equation in term of $w$
\[
 (\partial_t+\partial_r)(w_t - w_r) = \partial_t^2 w - \partial_r^2 w =  - r |u|^{p-1} u = - \frac{|w|^{p-1}w}{r^{p-1}}. 
\]
This then enable us to evaluate the variation of $w_t \pm w_r$ along characteristic lines $t \mp r = \hbox{Const}$ and obtain plentiful information about the asymptotic behaviour of solutions. Although the 3-dimensional case is indeed special, because we can not reduce the radial case of free wave equation $\partial_t^2 u - \Delta u = 0$ in other dimensions to an exact one-dimensional wave equation $\partial_t^2 w - \partial_r^2 w = 0$ in a similar way, we may still manage to generalize our results to higher dimensions. This will be the main topic of this current work. 

\subsection{The main idea}
Now let us explain how to generalize the 3D method to higher dimensions. Let $u$ be a radial solution to (CP1) with a finite energy. We may reduce the equation to a one-dimensional one by defining $w(r,t) = r^{\frac{d-1}{2}} u(r,t)$ and considering the equation that $w$ satisfies
\[
  (\partial_t+\partial_r)(w_t - w_r) = \partial_t^2 w - \partial_r^2 w = - \lambda_d r^{\frac{d-5}{2}} u - r^{\frac{d-1}{2}} |u|^{p-1} u.
\]
The constant $\lambda_d \doteq (d-1)(d-3)/4$ is determined by the dimension $d$ and will be frequently used in this work. As a result we apply the method of characteristic lines to obtain
\begin{align*}
 w_t(t_2\!-\!\eta,t_2) \!-\! w_r(t_2\!-\!\eta,t_2) & = w_t(t_1\!-\!\eta,t_1) \!-\! w_r(t_1\!-\!\eta,t_1) \\
 & \quad - \int_{t_1}^{t_2} \left(\lambda_d (t-\eta)^{\frac{d-5}{2}} u(t\!-\!\eta,t) + (t\!-\!\eta)^{\frac{d-1}{2}} |u|^{p-1} u(t\!-\!\eta,t)\right) dt
\end{align*}
for all $t_2>t_1>\eta$. Next we may verify that the integral above vanishes as $t_1,t_2 \rightarrow +\infty$ by the energy flux formula. Thus the function 
\[
 g_+(\eta) = \frac{1}{2} \lim_{t\rightarrow +\infty} \left[w_t(t-\eta,t) - w_r(t-\eta,t)\right].
\]
is always well-defined. This convergence helps to give the asymptotic behaviour of solution $w$ and $u$ as $t \rightarrow +\infty$. In general, the argument is similar to the 3-dimensional case. However, we have to overcome additional difficulties in higher dimensions $d\geq 4$. In fact, if we hope that $u$ scatters, i.e. there exists a free wave $\tilde{u}$ so that $\|(u,u_t)-(\tilde{u}, \tilde{u}_t)\|_{\dot{H}^1 \times L^2}\rightarrow 0$ as $t \rightarrow +\infty$, then $\tilde{w} = r^{(d-1)/2} \tilde{u}$ has to satisfies 
\begin{equation}
  \lim_{t\rightarrow +\infty} \left[\tilde{w}_t(t-\eta,t) - \tilde{w}_r(t-\eta,t)\right] =  \lim_{t\rightarrow +\infty} \left[w_t(t-\eta,t) - w_r(t-\eta,t)\right] = 2g_+(\eta) \label{to solve}
\end{equation}
Thus it is necessary to show the existence of such a free wave $\tilde{u}$ with prescribed asymptotic behaviour. In the 3-dimensional case, the function $\tilde{w}$ satisfies a simple equation $\partial_t^2 \tilde{w} - \partial_r^2 \tilde{w} = 0$. Therefore both $\tilde{w}$ and $\tilde{u}$ can be given explicitly in term of $g_+$:
\[
 \tilde{w}(r,t) = -  \int_{t-r}^{t+r} g_+(\eta) d\eta. \; \Rightarrow \; \tilde{u}(x,t) = - \frac{1}{|x|} \int_{t-|x|}^{t+|x|} g_+(\eta) d\eta. 
\]
In the higher dimensional case $d \geq 4$, however, $\tilde{w}$ satisfies the equation $\partial_t^2 \tilde{w} - \partial_r^2 \tilde{w} = - \lambda_d r^{-2} \tilde{w}$. The additional term $- \lambda_d r^{-2} \tilde{w}$ makes it much more difficult to solve $\tilde{w}$ from its asymptotic behaviour. Thus we will not solve $\tilde{u}$ explicitly. Instead we prove that given any suitable function $g_+(\eta)$ we may find a radial free wave $\tilde{u}$ so that $\tilde{w} = r^{(d-1)/2} \tilde{u}$ satisfies \eqref{to solve}. This is in fact the surjective property of the radiation field. Please see Section \ref{sec: radiation} for more details. 

\subsection{Main Results}

Now we give the statements of main theorems and then attach a few remarks. Throughout this paper we always assume $3\leq d\leq 6$, $1+4/(d-1)\leq p < 1+4/(d-2)$. The author would like to mention that the same idea still works in very high dimensions $d \geq 7$. We focus on the cases $3\leq d\leq 6$ in order to avoid technical difficulties in the local theory, as explained in Remark \ref{d36}. 

\begin{theorem}[Long time behaviour with finite energy] \label{main 1}
Assume that $3\leq d\leq 6$ and $1+4/(d-1)\leq p < 1+4/(d-2)$. Let $u$ be a radial solution to (CP1) with a finite energy $E$. Then there exists a free wave $\tilde{u}$ with energy\footnote{Energy conservation law of a free wave is well known: $\tilde{E} = (1/2)\|(\tilde{u}(\cdot,t), \tilde{u}_t(\cdot,t))\|_{\dot{H}^1 \times L^2}^{2}$ is independent of $t$.} $\tilde{E} \leq E$ so that 
\begin{itemize}
 \item [(a)] The solution $u$ scatters outside any given forward light cone $\{(x,t)\in \Rm^d \times \Rm: t - |x| = \eta\}$ in the positive time direction. Namely
 \[
  \lim_{t \rightarrow +\infty} \int_{|x|>t-\eta} \left(|\nabla u(x,t)-\nabla \tilde{u}(x,t)|^2 + |u_t(x,t)-\tilde{u}_t(x,t)|^2\right) dx = 0, \quad \forall \eta\in \Rm.
 \]
 \item [(b)] The solution $u$ scatters in the energy space $\dot{H}^1 \times L^2(\Rm^d)$, i.e. we have
 \[
  \lim_{t \rightarrow +\infty} \left\|(u(\cdot,t), u_t(\cdot,t))-(\tilde{u}(\cdot,t), \tilde{u}_t(\cdot,t))\right\|_{\dot{H}^1\times L^2} = 0,
 \]
 if and only if $\tilde{E}= E$.
\end{itemize}
The asymptotic behaviour of solution in the negative time direction is similar.
\end{theorem}

\begin{theorem}[Scattering with energy decay] \label{main 2}
Assume that $3\leq d\leq 6$ and $1+4/(d-1)\leq p < 1+4/(d-2)$. Let $u$ be a radial solution to (CP1) with initial data $(u_0,u_1) \in \dot{H}^1 \times L^2(\Rm^d)$ so that the following inequality holds for a constant $\kappa \geq \kappa_0(d,p) \doteq \frac{4-(d-2)(p-1)}{p+1}$:
 \[
  E_\kappa (u_0,u_1) \doteq \int_{\Rm^d} (1+|x|^\kappa)\left(\frac{1}{2}|\nabla u_0(x)|^2 + \frac{1}{2}|u_1(x)|^2 + \frac{1}{p+1}|u_0(x)|^{p+1}\right) dx < +\infty.
 \]
Then the solution $u$ scatters in the energy space $\dot{H}^1\times L^2$ in both two time directions. 
\end{theorem}

\begin{remark}
 Finite-energy free wave $\tilde{u}$ that satisfies conclusion part (a) of Theorem \ref{main 1} is unique. Because the difference $\bar{u}$ of two such free waves satisfies  
 \[
  \lim_{t \rightarrow +\infty} \int_{|x|>t-\eta} \left(|\nabla \bar{u}(x,t)|^2 + |\bar{u}_t(x,t)|^2\right) dx = 0, \quad \forall \eta \in \Rm,
 \]
thus has to be zero, according to Proposition \ref{asymptotic linear}. As a result, if a finite-energy radial solution $u$ to (CP1) does scatter in the positive time direction, then it has to approach the free wave $\tilde{u}$ given in Theorem \ref{main 1}.
\end{remark}

\begin{remark}
 We usually discuss the scattering of solutions in the critical Sobolev space $\dot{H}^{s_p} \times \dot{H}^{s_p-1}(\Rm^d)$. In this work, however, we use the energy space $\dot{H}^1 \times L^2$ instead. This is because our assumptions on the initial data are not sufficient to guarantee that $(u_0,u_1) \in \dot{H}^{s_p} \times \dot{H}^{s_p-1}(\Rm^d)$. For example, we may pick an arbitrarily small positive constant $\eps$ and choose a radial $C^\infty(\Rm^d)$ function $u_0$ so that 
\begin{align*}
 &u_0(x) = |x|^{-\frac{2(p+d+1)}{(p+1)^2}-\eps},& &|\nabla u_0(x)| \simeq |x|^{-\frac{2(p+d+1)}{(p+1)^2}-1-\eps},& &x \gg 1.&
\end{align*}
Then the initial data $(u_0,0)$ and $\kappa = \kappa_0(d,p)$ satisfy the conditions of both Theorem \ref{main 1} and Theorem \ref{main 2}. However, we also have $u_0 \notin L^{\frac{d(p-1)}{2}}(\Rm^d)$ if $\eps$ is sufficiently small because 
\[
 \frac{2(p+d+1)}{(p+1)^2} \cdot \frac{d(p-1)}{2} = \frac{(p+1)^2 + (d-2)(p-p_e(d))}{(p+1)^2} \cdot d < d,
\]
It immediately follows that $u_0 \notin \dot{H}^{s_p}(\Rm^d)$ since we have the Sobolev embedding $\dot{H}^{s_p}(\Rm^d) \hookrightarrow L^{\frac{d(p-1)}{2}}(\Rm^d)$. 
\end{remark}

\begin{remark}
 If $d=3$, the lower bound $\kappa_0(d,p)$ given in Theorem \ref{main 2} remains the same as in the 3-dimensional paper \cite{shenenergy}. But the endpoint case $\kappa = \kappa_0(d,p)$, which is prohibited in \cite{shenenergy}, is also allowed in this work. Although both works use the Morawetz estimates to deal with the energy that is located inside but far from the light cone, we adopt a more careful method of argument in this work thus improve the results slightly. 
 \end{remark}

\begin{remark}
 An application of the inward/outward energy theory as given in Shen \cite{shenhd} might slightly simplify the argument in this work. But the main result, i.e. the minimal decay rate of energy $\kappa_0(d,p)$ can not be further improved by the inward/outward energy theory.
\end{remark}

\subsection{Structure of this paper}
Before we conclude this section, we give the main topic of each section as below. Section 2 gives preliminary results. We collect necessary notations, technical lemmata, local theory, energy flux formula and Morawetz estimates in this section. Then in Section 3 we reduce the radial wave equation in higher dimension to one-dimensional wave equation and then utilize the method of characteristic lines to gather information about asymptotic behaviour of solutions. Next in Section 4 we show that given any solution $u$ to (CP1), there exists a free wave whose asymptotic behaviour is similar to that of $u$. Finally we prove the scattering results in the main theorems in the last section. 
\section{Preliminary Results}

\subsection{Notations}
We first introduce a few notations that will be used throughout this paper. 
\paragraph{Radial functions} Let $u(x)$ be a radial function defined in $\Rm^d$. We use the notation $u(r)$ for the value $u(x)$ at any point $x$ with $|x|=r$. Similarly we use the notation $u(r,t)$ for a spatially radial function $u(x,t)$. 
\paragraph{Sphere measure} In this work $\sigma_R$ represents the regular measure of the sphere $\{x\in \Rm^d: |x|=R\}$. We also define $c_d$ to be the area of the unit sphere $\mathbb{S}^{d-1}$. Thus we have the following identities for any radial function $f(x)$
\begin{align*}
 &\int_{|x|=r} f(x) d\sigma_r(x) = c_d r^{d-1} f(r);& &\int_{\Rm^d} f(x) dx = c_d \int_0^\infty f(r)r^{d-1} dr.&
\end{align*}
\paragraph{The $\lesssim$ symbol} The notation $A \lesssim B$ means that there exists a constant $c$, so that the inequality $A \leq cB$ holds. We may also put subscript(s) to indicate that the constant $c$ depends on the given subscript(s) but nothing else. In particular, the symbol $\lesssim_1$ is used if $c$ is an absolute constant. Similarly we use the notation $A \simeq B$ to indicates that there exists two constants $c_1,c_2$, so that $c_1 B \leq A \leq c_2 B$.
\paragraph{The $\doteq$ symbol} This symbol means that the formula in the right hand side is actually a definition of the notation in the left hand side.
\paragraph{Linear wave propagation operator} Let $(u_0,u_1)$ be initial data. We define $\mathbf{S}_L(u_0,u_1)$ to be the solution $u$ to free wave equation with initial data $(u_0,u_1)$. We may also specify a time $t$ and define $\mathbf{S}_L(t) (u_0,u_1) = (u(\cdot,t), u_t(\cdot,t))$ to be the data of solution $u$ at time $t$. 

\subsection{Technical Lemmata}
\begin{lemma}[Pointwise Estimate] \label{pointwise estimate 2}
 Assume $d\geq 3$. All radial $\dot{H}^1(\Rm^d)$ functions $u$ satisfy 
 \[
  |u(r)| \lesssim_d  r^{-\frac{d-2}{2}} \|u\|_{\dot{H}^1}, \qquad r>0.
 \]
 If $u$ also satisfies $u \in L^{p+1} (\Rm^d)$, then its decay is stronger as $r \rightarrow +\infty$.
 \[
  |u(r)| \lesssim_d r^{-\frac{2(d-1)}{p+3}} \|u\|_{\dot{H}^1}^{\frac{2}{p+3}} \|u\|_{L^{p+1}}^{\frac{p+1}{p+3}}, \qquad r>0.
 \]
\end{lemma}
\begin{proof}
 First of all, we have 
 \begin{align}
  |u(r_2) - u(r_1)| & = \left|\int_{r_1}^{r_2} u_r (r) dr\right| \leq \left(\int_{r_1}^{r_2} r^{-(d-1)} dr\right)^{1/2} \left(\int_{r_1}^{r_2} r^{d-1} |u_r|^2 dr\right)^{1/2} \label{variation of u 1}\\
  & \leq c_d^{-1/2}[r_1^{-(d-1)}(r_2-r_1)]^{1/2} \|u\|_{\dot{H}^1}. \label{variation of u}
 \end{align}
An $\dot{H}^1(\Rm^d)$ function $u$ must satisfy $\displaystyle \liminf_{r\rightarrow +\infty} |u(r)| = 0$. Thus we may make $r_2 \rightarrow +\infty$ in \eqref{variation of u 1} and obtain
\[
 |u(r_1)| \leq \left(\int_{r_1}^{\infty} r^{-(d-1)} dr\right)^{1/2} \left(\int_{r_1}^{\infty} r^{d-1} |u_r|^2 dr\right)^{1/2} \lesssim_d r_1^{-\frac{d-2}{2}} \|u\|_{\dot{H}^1}.
\]
This not only prove the first inequality in Lemma \ref{pointwise estimate 2} but also implies that there exists a small constant $c=c(d)$, so that the following inequalities hold for any fixed $r_0>0$.
 \begin{align*}
  \delta & \doteq \frac{c|u(r_0)|^2 r_0^{d-1}}{\|u\|_{\dot{H}^1}^2} \leq \frac{r_0}{2}; \\
  c_d^{-1/2}[(r_0-\delta)^{-(d-1)}\delta]^{1/2} \|u\|_{\dot{H}^1} & \leq c_d^{-1/2}(r_0/2)^{-\frac{d-1}{2}} \left(\frac{c|u(r_0)|^2 r_0^{d-1}}{\|u\|_{\dot{H}^1}^2}\right)^{1/2} \|u\|_{\dot{H}^1}
  \leq \frac{|u(r_0)|}{2}.
 \end{align*}
 We may use inequality \eqref{variation of u} and obtain the following estimate for all $r \in [r_0-\delta,r_0] \subseteq [r_0/2,r_0]$,
 \[
  |u(r_0)-u(r)| \leq c_d^{-1/2}[r^{-(d-1)}(r_0-r)]^{1/2} \|u\|_{\dot{H}^1}\leq c_d^{-1/2} [(r_0-\delta)^{-(d-1)}\delta]^{1/2} \|u\|_{\dot{H}^1} \leq |u(r_0)|/2.
 \]
 Thus for these $r$'s we have $|u(r)|\geq |u(r_0)|/2$. Next we use the $L^{p+1}$ norm
 \[
  r_0^{d-1} \delta |u(r_0)|^{p+1} \lesssim_d \int_{r_0-\delta}^{r_0} r^{d-1} \left(\frac{|u(r_0)|}{2}\right)^{p+1} dr \leq \int_{r_0-\delta}^{r_0} r^{d-1} |u(r)|^{p+1} dr \lesssim_d \|u\|_{L^{p+1}}^{p+1}.
 \]
 Finally we may plug the value of $\delta$ in the inequality above and obtain
 \[
  r_0^{d-1} \cdot \frac{c|u(r_0)|^2 r_0^{d-1}}{\|u\|_{\dot{H}^1}^2}\cdot |u(r_0)|^{p+1} \lesssim_d \|u\|_{L^{p+1}}^{p+1} 
  \Rightarrow |u(r_0)| \lesssim_d r_0^{-\frac{2(d-1)}{p+3}} \|u\|_{\dot{H}^1}^{\frac{2}{p+3}} \|u\|_{L^{p+1}}^{\frac{p+1}{p+3}}.
 \]
 \end{proof}

\begin{lemma} [See Lemma 2.1 of Shen \cite{shenhd}] \label{relation L2 uv nonradial}
 Let $u \in \dot{H}^{1}(\Rm^d)$ with $d \geq 3$. We define an operator
\[
 (\mathbf{L} u)(x) = r^{-\frac{d-1}{2}} \partial_r (r^{\frac{d-1}{2}} u) = \frac{x}{|x|}\cdot \nabla u(x) + \frac{d-1}{2} \cdot \frac{u(x)}{|x|}.
\]
Then we have the following identity ($\lambda\doteq (d-1)(d-3)/4$ is a constant)
\[
  \int_{\Rm^d} \left(\left|\mathbf{L} u\right|^2+ \lambda_d \cdot \frac{|u|^2}{|x|^2}\right) dx = \int_{\Rm^d} |u_r|^2 dx. 
\]
\end{lemma}

\subsection{Local theory and global existence}

\paragraph{Strichartz estimates} The key tools to develop a local theory are Strichartz estimates. The following version from Ginibre-Velo's work \cite{strichartz} is almost complete except for endpoint cases. Readers may refer to Keel-Tao \cite{endpointStrichartz} for endpoint Strichartz estimates. The author would like to mention that Ginibre-Velo \cite{strichartz} gives Strichartz estimates in both Besov and Sobolev spaces. Here we choose Sobolev spaces, which is more convenient to use in our argument.

\begin{proposition}[Strichartz estimates]\label{Strichartz estimates} 
 Let $2\leq q_1,q_2 \leq \infty$, $2\leq r_1,r_2 < \infty$ and $\rho_1,\rho_2,s\in \Rm$ be constants with
 \begin{align*}
  &\frac{2}{q_i} + \frac{d-1}{r_i} \leq \frac{d-1}{2},& &(q_i,r_i)\neq \left(2,\frac{2(d-1)}{d-3}\right),& &i=1,2;& \\
  &\frac{1}{q_1} + \frac{d}{r_1} = \frac{d}{2} + \rho_1 - s;& &\frac{1}{q_2} + \frac{d}{r_2} = \frac{d-2}{2} + \rho_2 +s.&
 \end{align*}
 Assume that $u$ is the solution to the linear wave equation
\[
 \left\{\begin{array}{ll} \partial_t u - \Delta u = F(x,t), & (x,t) \in \Rm^d \times [0,T];\\
 u|_{t=0} = u_0 \in \dot{H}^s; & \\
 \partial_t u|_{t=0} = u_1 \in \dot{H}^{s-1}. &
 \end{array}\right.
\]
Then we have
\begin{align*}
 \left\|\left(u(\cdot,T), \partial_t u(\cdot,T)\right)\right\|_{\dot{H}^s \times \dot{H}^{s-1}} & +\|D_x^{\rho_1} u\|_{L^{q_1} L^{r_1}([0,T]\times \Rm^d)} \\
 & \leq C\left(\left\|(u_0,u_1)\right\|_{\dot{H}^s \times \dot{H}^{s-1}} + \left\|D_x^{-\rho_2} F(x,t) \right\|_{L^{\bar{q}_2} L^{\bar{r}_2} ([0,T]\times \Rm^d)}\right).
\end{align*}
Here the coefficients $\bar{q}_2$ and $\bar{r}_2$ satisfy $1/q_2 + 1/\bar{q}_2 = 1$, $1/r_2 + 1/\bar{r}_2 = 1$. The constant $C$ does not depend on $T$ or $u$. 
\end{proposition}

\paragraph{Local theory} Assume that $3\leq d \leq 6$ and $p_c(d)\leq p<p_e(d)$. We use the notation $Y(I) = L^\frac{2p}{(d-2)p-d} L^{2p}(I \times \Rm^d)$ if $I$ is a time interval.  Given a fixed time $T>0$ and initial data $(u_0,u_1)\in \dot{H}^1 \times L^2$, we may introduce a transformation $\mathbf{P}: Y([0,T])\rightarrow Y([0,T])$ by defining $\mathbf{P} u$ to be the solution $U$ of the wave equation $\partial_t^2 U -\Delta U= F(u)$ with initial data $(u_0,u_1)$. Here we use the notation $F(u)=-|u|^{p-1}u$ for convenience. By Strichartz estimates there exist constants $C,C_1$, which are solely determined by $d,p$, so that
\begin{align*}
 \|\mathbf{P} u\|_{Y([0,T])} & \leq C \left(\|(u_0,u_1)\|_{\dot{H}^1 \times L^2} + \|F(u)\|_{L^1 L^2([0,T] \times \Rm^d)}\right)\\
 &  \leq C \left(\|(u_0,u_1)\|_{\dot{H}^1 \times L^2} + T^{\frac{(d+2)-(d-2)p}{2}} \|F(u)\|_{L^\frac{2}{(d-2)p-d} L^2([0,T] \times \Rm^d)}\right)\\
 & \leq C\left(\|(u_0,u_1)\|_{\dot{H}^1 \times L^2} + T^{\frac{(d+2)-(d-2)p}{2}} \|u\|_{Y([0,T])}^p\right).
\end{align*}
and 
\begin{align*}
 \|\mathbf{P}u -\mathbf{P}\tilde{u}\|_{Y([0,T])} & \leq C\|F(u)-F(\tilde{u})\|_{L^1 L^2([0,T]\times \Rm^d)}\\
 & \leq C T^{\frac{(d+2)-(d-2)p}{2}} \|F(u)-F(\tilde{u})\|_{L^\frac{2}{(d-2)p-d} L^2([0,T] \times \Rm^d)}\\
 & \leq C_1 T^{\frac{(d+2)-(d-2)p}{2}} \|u-\tilde{u}\|_{Y([0,T])} \left(\|u\|_{Y([0,T])}^{p-1} + \|\tilde{u}\|_{Y([0,T])}^{p-1}\right).
\end{align*}
As a result, there exists a constant $C(d,p)$, so that if we choose 
\begin{equation}
 T = C(d,p) \|(u_0,u_1)\|_{\dot{H}^1\times L^2}^{\frac{-2(p-1)}{(d+2)-(d-2)p}}, \label{lower bound of lifespan}
\end{equation}
then $\mathbf{P}$ becomes a contraction map from the complete distance space 
\begin{align*}
 &X = \{u: \|u\|_{Y([0,T])}\leq 2C\|(u_0,u_1)\|_{\dot{H}^1 \times L^2}\},& &d(u,\tilde{u}) = \|u-\tilde{u}\|_{Y([0,T])}&
\end{align*}
to itself. It immediately follows that $\mathbf{P}$ has a unique fixed-point in $X$. This proves the existence and uniqueness of local solution to (CP1) with initial data in the energy space $\dot{H}^1 \times L^2$. We summarize this local theory in Lemma \ref{lemma local in energy space} below. Please see Kapitanski \cite{loc1} and Lindblad-Sogge \cite{ls}, for instance, for more results and details about the local theory.

\begin{lemma} \label{lemma local in energy space}
Assume that $3\leq d\leq 6$ and $p \in [p_c(d),p_e(d))$. Let $(u_0,u_1) \in \dot{H}^1 \times L^2(\Rm^d)$ be initial data. Then the corresponding Cauchy problem (CP1) has a unique solution $u$ in the time interval $[0,T]$ with $(u(\cdot,t),u_t(\cdot,t)) \in C([0,T];\dot{H}^1 \times L^2(\Rm^d))$ and $u \in L^{\frac{2p}{(d-2)p-d}}L^{2p} ([0,T]\times \Rm^d)$. The minimal time length of existence $T$ here can be determined solely by the $\dot{H}^1 \times L^2$ norm of initial data, as given in \eqref{lower bound of lifespan}. 
\end{lemma}

\begin{remark} \label{d36}
 If $d \geq 7$, then $\frac{2p}{(d-2)p-d}<2$ when $p$ is slightly smaller than $p_e(d) = 1+4/(d-2)$. Thus Strichartz estimates do not apply to $L^{\frac{2p}{(d-2)p-d}}L^{2p}$ norms in this case. This is a technical difficulty we encounter in very high dimensions. 
\end{remark}

\paragraph{Global existence} If $u$ is a solution to (CP1) with a finite energy, then the minimal time of existence $T$ starting from any time $t_0$ has a uniform lower bound independent to $t_0$:
\[
 T = C(d,p) \|(u(\cdot,t_0), u_t(\cdot,t_0))\|_{\dot{H}^1 \times L^2}^{\frac{-2(p-1)}{(d+2)-(d-2)p}} \geq C(d,p) (2E)^{-\frac{p-1}{(d+2)-(d-2)p}},
\]
thanks to Lemma \ref{lemma local in energy space}. Thus $u$ is defined for all $t \in \Rm^+$. The same argument works in the negative time direction as well because the wave equation is time-reversible.

\begin{proposition}[Global existence] \label{global existence finite energy}
Assume that $3\leq d\leq 6$ and $p \in [p_c(d),p_e(d))$. If $u$ is a solution to (CP1) with a finite energy, then $u$ is defined for all time $t \in \Rm$. 
\end{proposition}

\subsection{Energy Flux Formula} \label{sec: energy flux}

Let $u$ be a finite-energy solution to the wave equation $\partial_t^2 u - \Delta u = \zeta |u|^{p-1} u$ in $\Rm^d$ with $d\geq 3$ and $p \in [p_c(d),p_e(d))$. The coefficient $\zeta=0,-1$ corresponds to the homogeneous linear and defocusing wave equation, respectively. Let us use the following notation for the energy inside a given region $\Sigma \in \Rm^d$ at time $t$
\[
 E(t;\Sigma) = \int_{\Sigma} \left(\frac{1}{2}|\nabla u(x,t)|^2 + \frac{1}{2}|u_t(x,t)|^2  - \frac{\zeta}{p+1}|u(x,t)|^{p+1}\right) dx.
\]
It is well-known that the following energy flux formula holds for all $t_2>t_1\geq \eta$.
\begin{align}
 E(t_2; B(0,t_2-\eta)) & - E(t_1; B(0,t_1-\eta)) \nonumber \\
 & = \frac{1}{\sqrt{2}} \int_{\Sigma(\eta; t_1,t_2)} \left(\frac{1}{2}|\nabla u|^2 + \frac{1}{2}|u_t|^2 + u_r u_t - \frac{\zeta}{p+1}|u|^{p+1} \right) dS \label{energy flux formula}
\end{align}
Here $B(0,t_i-\eta)\doteq \{x\in \Rm^d: |x|<t_i-\eta\}$ represents the ball centred at the origin with radius $t_i-\eta$ for $i \in \{1,2\}$. The surface $\Sigma(\eta;t_1,t_2)\doteq \{(x,t)\in \Rm^d \times \Rm: t-|x|=\eta, t_1\leq t\leq t_2\}$ is a part of the forward light cone. 
\paragraph{Finite speed of energy} Since the integrand is always nonnegative, $E(t;B(0,t-\eta))$ is always an increasing function $t \in [\eta,+\infty)$, i.e. the energy can never moves faster than the light speed. As a consequence $E(t;\{x\in \Rm^d: |x|>t-\eta\})$ is a decreasing function of $t \in [\eta,+\infty)$. This immediately gives the following limit
\begin{align*}
 \lim_{R \rightarrow +\infty} & \left\{\sup_{t\geq 0} \int_{|x|>t+R} \left(\frac{1}{2}|\nabla u(x,t)|^2 + \frac{1}{2}|u_t(x,t)|^2  - \frac{\zeta}{p+1}|u(x,t)|^{p+1}\right) dx  \right\} \\
 & = \lim_{R \rightarrow +\infty} \int_{|x|>R} \left(\frac{1}{2}|\nabla u_0(x)|^2 + \frac{1}{2}|u_1(x)|^2  - \frac{\zeta}{p+1}|u_0(x)|^{p+1}\right) dx = 0.
\end{align*}
One may also consider the energy flux through backward light cones $|x|+t=s$, then prove the monotonicity of $E(t; B(0,s-t))$ and $E(t;\{x\in \Rm^d: |x|>s-t\})$ in the same manner.
\begin{proposition} \label{energy flux}
Assume that $\zeta \in \{0,-1\}$. Let $u$ be a solution to the wave equation $\partial_t^2 u - \Delta u = \zeta |u|^{p-1} u$ with a finite energy. Then given any $\eta \in \Rm$, $E(t; B(0,t-\eta))$ is an increasing function of $t \in [\eta,+\infty)$; $E(t; \{x\in \Rm^d: |x|>t-\eta\})$ is a decreasing function of $t \in [\eta,+\infty)$. Similarly given any $s\in \Rm$, $E(t; B(0,s-t))$ is a decreasing function of $t \in (-\infty,s]$; $E(t; \{x\in \Rm^d: |x|>s-t\})$ is an increasing function of $t\in (-\infty,s]$. We also have the following limit
 \[
  \lim_{R \rightarrow +\infty} \left\{\sup_{t\geq 0} \int_{|x|>t+R} \left(\frac{1}{2}|\nabla u(x,t)|^2 + \frac{1}{2}|u_t(x,t)|^2  - \frac{\zeta}{p+1}|u(x,t)|^{p+1}\right) dx  \right\} = 0.
 \]
\end{proposition}
\paragraph{Surface integral estimates} Next we observe that the left hand of \eqref{energy flux formula} is smaller or equal to the energy $E$, let $t_1=\eta$, $t_2\rightarrow +\infty$ and obtain an inequality
\[
 \frac{1}{\sqrt{2}} \int_{t-|x|=\eta} \left(\frac{1}{2}|\nabla u|^2 + \frac{1}{2}|u_t|^2 + u_r u_t - \frac{\zeta}{p+1}|u|^{p+1} \right) dS \leq E.
\] 
If we define $\bar{u} (x)= u(x, |x|+\eta)$, then we may apply Hardy's inequality and obtain
\begin{align*}
 \frac{1}{\sqrt{2}} \int_{t-|x|=\eta} \frac{|u(x,t)|^2}{|x|^2} dS & = \int_{\Rm^d} \frac{|\bar{u}(x)|^2}{|x|^2} dx \lesssim_d \frac{1}{2}\int_{\Rm^d} |\nabla \bar{u}(x)|^2 dx \\
 & = \frac{1}{\sqrt{2}} \int_{t-|x|=\eta} \left(\frac{1}{2}|\nabla u|^2 + \frac{1}{2}|u_t|^2 + u_r u_t\right) dS \leq E.
\end{align*}
In summary we have
\begin{proposition}[Boundedness of energy flux] \label{energy flux integral}
 Assume $d \geq 3$ and $\zeta \in \{-1,0\}$. Let $u$ be a solution to the wave equation $\partial_t^2 u - \Delta u = \zeta |u|^{p-1} u$ with a finite energy $E$. Then we have the following uniform upper bounds on the surface integrals over light cones
 \begin{align*}
  \frac{1}{\sqrt{2}} \int_{t-|x|=\eta} \left(\frac{1}{2}|\nabla u|^2 + \frac{1}{2}|u_t|^2 + u_r u_t - \frac{\zeta}{p+1}|u|^{p+1} \right) dS & \leq E;\\
  \frac{1}{\sqrt{2}} \int_{t-|x|=\eta} \frac{|u(x,t)|^2}{|x|^2} dS & \lesssim_d E.
 \end{align*}
\end{proposition}

\subsection{Morawetz estimates}

The following Morawetz estimate was given by Perthame and Vega in their work \cite{benoit}. It provides valuable information about the energy distribution of solutions to defocusing wave equation. A slightly stronger version of Morawetz estimates can be found in the author's recent work \cite{shenhd}. We assume $d\geq 3$ and $p\in [p_c(d),p_e(d)]$ in this subsection.

\begin{proposition}[Morawetz estimates] Let $u$ be a solution to (CP1) defined in a time interval $[0,T]$ with a finite energy $E$. Then we have the following inequality for any $R>0$.  
\begin{align}
 & \frac{1}{2R}\int_{0}^T \!\int_{|x|<R}\left(|\nabla u|^2+|u_t|^2+\frac{(d\!-\!1)(p\!-\!1)\!-\!2}{p+1}|u|^{p+1}\right) dx dt \nonumber\\
 &\quad  + \frac{d-1}{4R^2} \int_{0}^T \!\int_{|x|=R} |u|^2 d\sigma_R(x) dt + \frac{(d-1)(p-1)}{2(p+1)} \int_{0}^T\! \int_{|x|>R} \frac{|u|^{p+1}}{|x|} dx dt  \leq 2E. \label{morawetz1}
\end{align}
\end{proposition}
\begin{remark}
 Perthame and Vega write the nonlinear term of wave equation as $-|u|^p u$. In addition, the energy they define is twice as much as ours. Thus the notations $p$ and $E$ represent slight different constants in their works. This explains why the coefficients of the Morawetz inequality in their work look different from ours. We also ignore two other nonnegative terms in the left hand side that are irrelevant to our argument in this work. 
\end{remark}

\paragraph{Energy distribution} We have already known that all finite-energy solutions to (CP1) are globally defined in time. Thus we may substitute the upper limit of integrals in inequality \eqref{morawetz1} by $+\infty$. By energy conservation law we may also substitute the lower limit by $-\infty$. 
\begin{align}
 & \frac{1}{2R}\int_{-\infty}^\infty \!\int_{|x|<R}\left(|\nabla u|^2+|u_t|^2+\frac{(d\!-\!1)(p\!-\!1)\!-\!2}{p+1}|u|^{p+1}\right) dx dt \nonumber\\
 &\quad  + \frac{d-1}{4R^2} \int_{-\infty}^\infty \!\int_{|x|=R} |u|^2 d\sigma_R(x) dt + \frac{(d-1)(p-1)}{2(p+1)} \int_{-\infty}^\infty \! \int_{|x|>R} \frac{|u|^{p+1}}{|x|} dx dt  \leq 2E. \label{morawetz2}
\end{align}
\noindent Because we assume $p \geq 1+4/(d-1)$, we have $\frac{(d-1)(p-1)-2}{p+1} \geq \frac{2}{p+1}$. As a result we have
\begin{align*}
  \frac{1}{2R}\int_{-\infty}^\infty \!\int_{|x|<R}\left(|\nabla u|^2+|u_t|^2+\frac{2}{p+1}|u|^{p+1}\right) dx dt \leq 2E.
\end{align*}
Thus
\begin{align*}
 \int_{-\infty}^\infty \!\int_{|x|<R} &\left(\frac{|\nabla u|^2}{2}+\frac{|u_t|^2}{2}+\frac{|u|^{p+1}}{p+1}\right) dx dt\\
 & \leq 2RE = \int_{-R}^R \int_{\Rm^d} \left(\frac{|\nabla u|^2}{2}+\frac{|u_t|^2}{2}+\frac{|u|^{p+1}}{p+1}\right) dx dt.
\end{align*}
We may subtract $\int_{-R}^R \int_{|x|<R} \left(\frac{|\nabla u|^2}{2}+\frac{|u_t|^2}{2}+\frac{|u|^{p+1}}{p+1}\right) dx dt$ from both sides and obtain
\begin{corollary} \label{energy distribution}
 Let $u$ be a finite-energy solution to (CP1). Then we have the inequality
 \begin{align*}
 \int_{|t|>R} \!\int_{|x|<R} &\left(\frac{|\nabla u|^2}{2}+\frac{|u_t|^2}{2}+\frac{|u|^{p+1}}{p+1}\right) dx dt
 \leq \int_{-R}^R \int_{|x|>R} \left(\frac{|\nabla u|^2}{2}+\frac{|u_t|^2}{2}+\frac{|u|^{p+1}}{p+1}\right) dx dt.
\end{align*}
\end{corollary}

\paragraph{lower limit of $\|u\|_{L^{p+1}}$} We let $R \rightarrow 0^+$ in the Morawetz inequality \eqref{morawetz2} and obtain
\[
  \int_{-\infty}^\infty \int_{\Rm^d} \frac{|u|^{p+1}}{|x|} dx dt \lesssim_{d,p} E.
\]
This is the most widely used form of Morawetz estimates. It immediately follows that 
\begin{corollary} \label{limit of potential}
 If $u$ is a finite-energy solution to (CP1), then $\displaystyle \liminf_{t\rightarrow +\infty} \int_{\Rm^d} |u(x,t)|^{p+1} dx = 0$.
\end{corollary}
\begin{proof}
 Given any $R>0$, we have 
 \[
  \int_{0}^\infty \left(\frac{1}{t+R} \int_{|x|<t+R} |u(x,t)|^{p+1} dx\right) dt \leq \int_{0}^\infty \int_{\Rm^d} \frac{|u(x,t)|^{p+1}}{|x|} dx dt \lesssim_{d,p} E.
 \]
 This implies that $\displaystyle \liminf_{t\rightarrow +\infty} \int_{|x|<t+R} |u(x,t)|^{p+1} dx = 0$. We then combine this lower limit with Proposition \ref{energy flux} to finish the proof.
 \end{proof}

\subsection{Asymptotic behaviour of free waves}
Before we conclude this section, we give a lemma describing the asymptotic behaviour of free waves. 
\begin{lemma} \label{asymptotic linear}
 Assume $d\geq 3$. Let $u$ be a solution to the free wave equation $\partial_t^2 u - \Delta u =0$ with initial data $\dot{H}^1\times L^2(\Rm^d)$. Then we have the limits
 \begin{align*}
  &\lim_{t \rightarrow +\infty} \int_{\Rm^d} \frac{|u(x,t)|^2}{|x|^2} dx = 0,& \\
  &\lim_{\eta \rightarrow +\infty} \sup_{t>\eta} \left\{\int_{|x|<t-\eta} \left(|\nabla u(x,t)|^2 + |u_t(x,t)|^2\right) dx\right\} = 0.&
 \end{align*}
\end{lemma}
\begin{proof} 
 These results are classical. We give a proof here for readers' convenience.  First of all, we have 
 \begin{align*}
  \int_{\Rm^d} \frac{|u(x,t)|^2}{|x|^2} dx & \lesssim_d \|u(\cdot,t)\|_{\dot{H}^1}^2 \leq \|(u_0,u_1)\|_{\dot{H}^1 \times L^2}^2, \\
  \sup_{t>\eta} \left\{\int_{|x|<t-\tau} \left(|\nabla u(x,t)|^2 + |u_t(x,t)|^2\right) dx\right\} & \leq \|(u(\cdot,t), u_t(\cdot,t))\|_{\dot{H}^1 \times L^2}^2 = \|(u_0,u_1)\|_{\dot{H}^1 \times L^2}^2,
 \end{align*}
 by Hardy's inequality and the unitary property of the linear wave propagation operator, respectively. Therefore we also need to prove the limits for initial data $(u_0,u_1)$ which are smooth and compactly supported. Because these initial data are dense in the space $\dot{H}^1 \times L^2$. Given such initial data $(u_0,u_1)\in C_0^\infty (B(0,r_0)) \subset C_0^\infty (\Rm^d)$, it is well known that $u = \mathbf{S}_L(u_0,u_1)$ satisfies a uniform decay estimate $|u(x,t)| \leq C |t|^{-\frac{d-1}{2}}$. A simple calculation shows 
 \begin{align*}
  \int_{|x|<t+r_0} \frac{|u(x,t)|^2}{|x|^2} dx \lesssim_d \int_{|x|<t+r_0} \frac{C^2 t^{-(d-1)}}{|x|^2} dx \lesssim_d \frac{C^2 (t+r_0)^{d-2}}{|t|^{d-1}}.
 \end{align*}
 We also have $u(x,t)\equiv 0$ if $|x|>t+r_0$ by finite speed of propagation. This immediately prove the first limit. The second one requires a more careful analysis. If $d$ is odd, we have $u(x,t)\equiv 0$ as long as $|x|<t-r_0$, by strong Huygens' principle. Thus in this case we have 
 \[
  \int_{|x|<t-\eta} \left(|\nabla u(x,t)|^2 + |u_t(x,t)|^2\right) dx = 0, \quad \forall t\geq \eta
 \]
 for all $\eta > r_0$. This proves the odd dimensional case. If $d$ is even, however, we have to recall the formula of solution to free wave equation (see section 2.4 of \cite{pdeevans}, for example) in details and obtain
 \begin{align*}
  u(x,t) = & C_d \left(\frac{\partial}{\partial_t}\right)\left(\frac{1}{t} \frac{\partial}{\partial_t}\right)^\frac{d-2}{2} \int_{B(0,r_0)} \frac{u_0(y)}{(t^2-|y-x|^2)^{1/2}} dy\\
  & \qquad + C_d \left(\frac{1}{t} \frac{\partial}{\partial_t}\right)^\frac{d-2}{2} \int_{B(0,r_0)} \frac{u_1(y)}{(t^2-|y-x|^2)^{1/2}} dy,
 \end{align*}
 for all $(x,t) \in \Omega = \{(x,t): |x|<t-r_0\}$. A simple differentiation shows that 
 \begin{align*}
 \nabla u(x,t) = & C_d \left(\frac{\partial}{\partial_t}\right)\left(\frac{1}{t} \frac{\partial}{\partial_t}\right)^\frac{d-2}{2} \int_{B(0,r_0)} \frac{u_0(y)\cdot (x-y)}{(t^2-|y-x|^2)^{3/2}} dy\\
  & \qquad + C_d \left(\frac{1}{t} \frac{\partial}{\partial_t}\right)^\frac{d-2}{2} \int_{B(0,r_0)} \frac{u_1(y)\cdot (x-y)}{(t^2-|y-x|^2)^{3/2}} dy.
 \end{align*}
 Although the expression becomes more and more complicated after we differentiate in $t$ multiple times, each term involved in the calculation must be a constant multiple of
 \[
  \int_{B(0,r_0)} \frac{t^k u_j(y)\cdot(x-y)}{(t^2-|y-x|^2)^{(2n+1)/2}} dy, \quad j \in \{0,1\},\; k\in \mathbb{Z},\; n \in \mathbb{Z}^+,
 \]
When we differentiate in $t$, new terms are introduced by either deducting one from the exponent $k$ or multiplying the integrand by $t/(t^2-|y-x|^2) \geq 1/t$, both up to a constant multiple. Thus the worst terms in the expression of $\nabla u(x,t)$ are constant multiples of
\begin{align*}
 &\int_{B(0,r_0)} \frac{u_0(y)\cdot (x-y)}{(t^2-|y-x|^2)^{3/2}}\cdot t^{-\frac{d-2}{2}}\left(\frac{t}{t^2-|y-x|^2}\right)^{d/2} dy,&\\
  &\int_{B(0,r_0)} \frac{u_1(y)\cdot (x-y)}{(t^2-|y-x|^2)^{3/2}}\cdot t^{-\frac{d-2}{2}}\left(\frac{t}{t^2-|y-x|^2}\right)^{\frac{d-2}{2}} dy.&
\end{align*}
This gives an estimate for all $(x,t)\in \Omega$
\begin{equation*}
 |\nabla u(x,t)|\lesssim_d \frac{r_0^d \|u_0\|_{L^\infty}}{t^{\frac{d-1}{2}}(t-r_0-|x|)^{\frac{d+3}{2}}} + \frac{r_0^d \|u_1\|_{L^\infty}}{t^{\frac{d-1}{2}}(t-r_0-|x|)^{\frac{d+1}{2}}},
\end{equation*} 
because we have $t^2-|y-x|^2 = (t+|y-x|)(t-|y-x|) \geq t (t-r_0-|x|)$ and $|y-x|<t$. A similar argument shows that $|u_t(x,t)|$ can be dominated by the same upper bound for all $(x,t) \in \Omega$. We may substitute both $|\nabla u(x,t)|$ and $u_t(x,t)$ by their upper bound, integrate and obtain
\begin{align*}
 \int_{|x|<t-\eta} (|\nabla u(x,t)|^2 +|u_t(x,t)|^2) dx & \lesssim_d \frac{r_0^{2d} \|u_0\|_{L^\infty}^2}{(\eta-r_0)^{d+2}} + \frac{r_0^{2d} \|u_1\|_{L^\infty}^2} {(\eta-r_0)^{d}}
\end{align*}
for all $\eta > r_0$. This vanishes as $\eta \rightarrow +\infty$. 
\end{proof}

\section{Method of Characteristic Lines} \label{sec: reduction to 1D}

We will rewrite the wave equation with radial initial data as a one-dimensional wave equation, in order to take full advantage of the radial assumption. In the argument below we proceed as though the solution is sufficiently smooth. We may apply smooth approximation techniques to deal with general radial solutions that are not sufficiently smooth. For convenience we first introduce a few notations.
\begin{definition} \label{def of wv}
 Let $u(x,t)$ be a spatially radial function with $(u,u_t) \in C(\Rm; \dot{H}^1 \times L^2(\Rm^d))$. We define a few functions for $(r,t)\in \Rm^+ \times \Rm$:
 \begin{align*}
  &w(r,t) = r^{\frac{d-1}{2}} u(r,t);& &&\\
  &v_+(r,t) = w_t(r,t) - w_r(r,t);& &v_-(r,t) = w_t(r,t) + w_r(r,t).&
 \end{align*}
\end{definition}
\noindent It is clear that $|v_+|^2 + |v_-|^2 = 2r^{d-1}(|\mathbf{L}u|^2+|u_t|^2)$. According to Lemma \ref{relation L2 uv nonradial}, we have
\begin{lemma} \label{relation L2 uv}
Let $u, v_+, v_-$ be as in Definition \ref{def of wv}. Then for any given $t$ we have
\begin{align*}
  \int_{\Rm^d} (|\nabla u|^2 +|u_t|^2) dx = \lambda_d \int_{\Rm^d} \frac{|u|^2}{|x|^2} dx + \frac{c_d}{2}\int_{0}^{\infty} (|v_+|^2 + |v_-|^2) dr. 
\end{align*}
Thus we have $v_+(\cdot,t), v_-(\cdot,t) \in L^2(\Rm^+)$ for all $t$. 
\end{lemma}
\paragraph{Variation of $v_+, v_-$} Let us consider a radial solution $u(x,t)$ with a finite energy to either linear or nonlinear wave equation
\[
 \partial_t^2 u - \Delta u = \zeta |u|^{p-1} u, \qquad \zeta \in \{-1,0\}.\quad \hbox{(WAVE)}
\]
The coefficient $\zeta = -1$ corresponds to the defocusing case (CP1); while $\zeta =0$ corresponds to the homogenous linear wave equation. A simple calculation verifies the identity
\begin{equation}
 (\partial_t^2 - \partial_r^2) (r^{\frac{d-1}{2}} u) = r^{\frac{d-1}{2}}(\partial_t^2 - \Delta) u - \lambda_d r^{\frac{d-5}{2}} u. \label{transformation uw}
\end{equation}
Therefore $w$, $v_\pm$ defined above satisfy the equation
\begin{align*}
 (\partial_t \pm \partial_r) v_\pm (r,t) = \partial_t^2 w - \partial_r^2 w = - \lambda_d r^{\frac{d-5}{2}} u + \zeta r^{\frac{d-1}{2}} |u|^{p-1} u. 
\end{align*}
This immediately gives variation of $v_\pm$ along characteristic lines $t \pm r = \hbox{Const}$. 
\begin{lemma} \label{variation of v}
 Let $u$ be a radial solution to (WAVE) with a finite energy. Then the function $v_+, v_-$ defined above satisfy
 \begin{align*}
  &v_+(t_2-\eta, t_2) - v_+(t_1-\eta, t_1)  = \int_{t_1}^{t_2} f(t-\eta,t) dt,& &t_2>t_1>\eta;& \\
  &v_-(s-t_2,t_2) - v_-(s-t_1,t_1)  = \int_{t_1}^{t_2} f(s-t,t) dt,& &t_1<t_2<s.&
 \end{align*} 
 Here the function $f(r,t)$ is defined by
 \[
  f(r,t) = - \lambda_d r^{\frac{d-5}{2}} u(r,t) + \zeta r^{\frac{d-1}{2}} |u|^{p-1} u(r,t).
 \]
\end{lemma}

\paragraph{Upper bounds of the integral} Now let us find an upper bound of the integral of $f$ above. We first recall Lemma \ref{energy flux integral} and obtain
\begin{align*}
 \int_{\eta}^\infty (t-\eta)^{d-3} |u(t-\eta,t)|^2 dt  \lesssim_d \int_{|x|=t-\eta} \frac{|u(x,t)|^2}{|x|^2} dS \lesssim_d E.
\end{align*}
This immediately gives us the following upper bound of integral along characteristic lines
\begin{align*}
 &\;\;\; \int_{t_1}^{t_2} (t-\eta)^{\frac{d-5}{2}} |u(t-\eta,t)| dt\\
  & \leq  \left\{\int_{t_1}^{t_2}\left[(t-\eta)^{\frac{d-3}{2}} |u(t-\eta,t)| \right]^2 dt\right\}^{1/2}
  \left\{\int_{t_1}^{t_2}\left[(t-\eta)^{-1}\right]^2 dt\right\}^{1/2}\\
  &\leq \left\{\int_{\eta}^\infty (t-\eta)^{d-3} |u(t-\eta,t)|^{2} dt\right\}^{1/2} (t_1-\eta)^{-1/2}\\
  &\lesssim_d  E^{1/2} (t_1-\eta)^{-1/2}.
\end{align*}
If $\zeta=-1$, i.e. $u$ solves a defocusing wave equation, Lemma \ref{energy flux integral} gives us another integral estimate
\[
  \int_{\eta}^\infty (t-\eta)^{d-1} |u(t-\eta,t)|^{p+1} dt  \lesssim_d \int_{|x|=t-\eta} |u(x,t)|^{p+1} dS \lesssim_1 E.
\]
This deals with the integral of $\zeta r^{\frac{d-1}{2}} |u|^{p-1} u(r,t)$. 
\begin{align*}
 &\;\;\; \int_{t_1}^{t_2} (t-\eta)^{\frac{d-1}{2}} |u(t-\eta,t)|^p dt\\
  & \leq \left\{\int_{t_1}^{t_2}\left[(t-\eta)^{\frac{(d-1)p}{p+1}}|u(t-\eta,t)|^p\right]^{\frac{p+1}{p}}dt\right\}^{\frac{p}{p+1}}
 \left\{\int_{t_1}^{t_2}\left[(t-\eta)^{-\frac{(d-1)(p-1)}{2(p+1)}}\right]^{p+1} dt\right\}^{\frac{1}{p+1}}\\
 & \leq \left\{\int_{\eta}^\infty (t-\eta)^{d-1} |u(t-\eta,t)|^{p+1} dt\right\}^{\frac{p}{p+1}}
  \left\{\int_{t_1}^{t_2} (t-\eta)^{-\frac{(d-1)(p-1)}{2}} dt\right\}^{\frac{1}{p+1}}\\
&  \lesssim_d E^{\frac{p}{p+1}} (t_1-\eta)^{-\frac{(d-1)(p-1)-2}{2(p+1)}}.
\end{align*}
Our assumption $p \in [p_c(d), p_e(d))$ guarantees that $(d-1)(p-1) \geq 4$. One can also consider the integral of $f$ along characteristic lines $t+r=s$. A similar upper bound can found in the same manner. We may combine these estimates with Lemma \ref{variation of v} to obtain
\begin{proposition} \label{upper bound of variation v}
 Let $u$ be a radial solution to (WAVE) with a finite energy $E$.  Then we have 
 \begin{align*} 
   \left|v_+(t_2-\eta,t_2)-v_+(t_1-\eta, t_1)\right| &\lesssim_d E^{1/2} (t_1-\eta)^{-1/2} + |\zeta| E^{\frac{p}{p+1}} (t_1-\eta)^{-\beta(d,p)} ;\\
   \left|v_-(s-t_2,t_2)-v_-(s-t_1, t_1)\right| & \lesssim_d E^{1/2} (s-t_2)^{-1/2} + |\zeta|E^{\frac{p}{p+1}} (s-t_2)^{-\beta(d,p)};
 \end{align*}
 for all $\eta<t_1<t_2<s$. The decay rate $\beta(d,p) \doteq \frac{(d-1)(p-1)-2}{2(p+1)}$ always satisfies $0<\beta(d,p) < \frac{1}{2}$ by our assumption $p_c(d) \leq p<p_e(d)$. 
\end{proposition}
\paragraph{The limits of $v_\pm$} By Lemma \ref{relation L2 uv}, we have $\|v_+(t-\eta,t)\|_{L_\eta^2((-\infty,t))}^2 \leq 4E/c_d$ for all time $t$. Proposition \ref{upper bound of variation v} implies that given any $\eta_1<\eta_2$, the functions $v_+(t-\eta,t)$ converges in the space $L_\eta^{2}([\eta_1,\eta_2])$ as $t\rightarrow +\infty$. Therefore there exists a function $g_+(\eta) \in L^2(\Rm)$ with $\|g_+\|_{L^2(\Rm)}^2 \leq E/c_d$ so that 
\[
 v_+(t-\eta,t) \rightarrow 2g_+(\eta) \quad \hbox{in} \; L_{loc}^2 (\Rm),\quad \hbox{as}\; t \rightarrow +\infty.
\]
The asymptotic behaviour of $v_-$ is similar as $t \rightarrow -\infty$. In summary we have
\begin{proposition} \label{introduction of g}
 Let $u$ be a radial solution to (WAVE) with a finite energy $E$. Then there exists two unique functions $g_+, g_-$ with $\|g_+\|_{L^2(\Rm)}^2, \|g_-\|_{L^2(\Rm)}^2 \leq E/c_d$ so that we have the following local $L^2$ convergence
 \begin{align*}
  v_+(t-\eta,t) \rightarrow 2g_+(\eta) \quad \hbox{in} \; L_{loc}^2 (\Rm),\quad \hbox{as}\; t \rightarrow +\infty;\\
  v_-(s-t,t) \rightarrow 2g_-(s) \quad \hbox{in}\; L_{loc}^2 (\Rm),\quad \hbox{as}\; t \rightarrow -\infty.
 \end{align*}
\end{proposition}
\begin{definition} \label{def of T}
 Let $\dot{H}_{rad}^1\times L_{rad}^2(\Rm^d)$ be the space of radial $\dot{H}^1 \times L^2$ functions. We may define a bounded linear operator $\mathbf{T}_+$ from this space to $L^2(\Rm)$. Given any $(u_0,u_1) \in \dot{H}_{rad}^1\times L_{rad}^2(\Rm^d)$, the function $u = \mathbf{S}_L (u_0,u_1)$ is a radial solution to homogenous linear wave equation. We define 
 \[
  \mathbf{T}_+ (u_0,u_1) = g_+.
 \]
 The function $g_+$ is defined as in Proposition \ref{introduction of g}. 
\end{definition}
\subsection{Convergence rate of nonlinear solution} \label{sec: convergence rate}
Assume that $u$ is a radial to (CP1) with a finite energy $E$. Now let us consider the convergence rate of $v_\pm$ to $g_\pm$.  Let us recall Proposition \ref{upper bound of variation v} and let $t_2 \rightarrow +\infty$ in the first inequality 
\[
 \left|2g_+(\eta)-v_+(t -\eta, t)\right| \lesssim_{d,E} (t-\eta)^{-\beta(d,p)}, \quad \eta < t-1. 
\]
We apply a change of variable $r = t - \eta$ and rewrite this in the form 
\[
 \left|v_+(r, t) - 2g_+(t-r)\right| \lesssim_{d,E} r^{-\beta(d,p)}, \quad r>1. 
\]
Similarly we have
\[
 \left|v_-(r, t) - 2g_-(t+r)\right| \lesssim_{d,E} r^{-\beta(d,p)}, \quad r>1.
\]
These immediately gives the following upper limits for all constants $c,R>0$ and $\gamma \in [0,2\beta(d,p))$:
 \begin{align*}
  \limsup_{t \rightarrow + \infty} \int_{t-c\cdot t^{\gamma}}^{t+R} \left(\left|v_+(r, t) - 2g_+(t-r)\right|^2 + \left|v_-(r, t) - 2g_-(t+r)\right|^2\right) dr & = 0;\\
  \limsup_{t \rightarrow + \infty} \int_{t-c\cdot t^{2\beta(d,p)}}^{t+R} \left(\left|v_+(r, t) - 2g_+(t-r)\right|^2 + \left|v_-(r, t) - 2g_-(t+r)\right|^2\right) dr & \lesssim_{d,E} c;
 \end{align*}
We may ignore $g_-(t+r)$ in the upper limits above because
\begin{align*}
 \lim_{t\rightarrow +\infty} \int_{0}^\infty |g_-(t+r)|^2 dr = \lim_{t\rightarrow +\infty} \int_{t}^\infty |g_-(s)|^2 ds = 0.
\end{align*}
Next we recall $v_\pm = w_t \mp w_r$ and rewrite the upper limits above in terms of $w$
 \begin{align*}
  \limsup_{t \rightarrow + \infty} \int_{t-c\cdot t^{\gamma}}^{t+R} \left(\left|w_r (r, t) + g_+(t-r)\right|^2 + \left|w_t(r, t) - g_+(t-r)\right|^2\right) dr & = 0;\\
  \limsup_{t \rightarrow + \infty} \int_{t-c\cdot t^{2\beta(d,p)}}^{t+R} \left(\left|w_r(r, t) + g_+(t-r)\right|^2 + \left|w_t(r, t) - g_+(t-r)\right|^2\right) dr & \lesssim_{d,E} c;
 \end{align*}
Finally we utilize the identities $r^{\frac{d-1}{2}} u_r = w_r - (d-1) r^{\frac{d-3}{2}} u/2$, $r^{\frac{d-1}{2}} u_t = w_t$ and a direct consequence of the pointwise estimate $|u(r,t)| \lesssim_{d,E} r^{-\frac{2(d-1)}{p+3}}$ (See Lemma \ref{pointwise estimate 2})
 \begin{align*}
  \int_{t/2}^\infty r^{d-3} |u(r,t)|^2 dr  \lesssim_{d,E} t^{-\frac{(d+2)-(d-2)p}{p+3}}\quad \Rightarrow \quad \lim_{t\rightarrow +\infty} \int_{t/2}^\infty \left|r^{\frac{d-3}{2}} u(r,t)\right|^2 dr = 0.
 \end{align*}
 to conclude
\begin{proposition} \label{middle convergence defocusing}
 Let $u$ be a radial solution to (CP1) with a finite energy $E$. Given any constants $c, R>0$ and $\gamma\in [0,2\beta(d,p))$, we have
 \begin{align*}
 \limsup_{t \rightarrow + \infty} \int_{t-c\cdot t^{\gamma}}^{t+R} \left(\left|r^{\frac{d-1}{2}}u_r(r,t) + g_+(t-r)\right|^2 + \left|r^{\frac{d-1}{2}} u_t(r,t) - g_+(t-r)\right|^2\right) dr & = 0;\\
  \limsup_{t \rightarrow + \infty} \int_{t-c\cdot t^{2\beta(d,p)}}^{t+R} \left(\left|r^{\frac{d-1}{2}}u_r(r,t) + g_+(t-r)\right|^2 + \left|r^{\frac{d-1}{2}} u_t(r,t) - g_+(t-r)\right|^2\right) dr & \lesssim_{d,E} c;
 \end{align*}
\end{proposition}
\subsection{Global $L^2$ convergence of free wave}
The same argument as in Subsection \ref{sec: convergence rate} also works for radial free waves $u$. In fact, the convergence rate is even better for large $r$:
\[
  \left|v_+(r, t) - 2g_+(t-r)\right| +  \left|v_-(r, t) - 2g_-(t+r)\right| \lesssim_d E^{1/2} r^{-1/2}.
\]
As a result, we have the following limit for all $R>0$ and $0<\gamma<1$
 \begin{equation} \label{middle estimate linear}
 \limsup_{t \rightarrow + \infty} \int_{t- t^{\gamma}}^{t+R} \left(\left|r^{\frac{d-1}{2}}u_r(r,t) + g_+(t-r)\right|^2 + \left|r^{\frac{d-1}{2}} u_t(r,t) - g_+(t-r)\right|^2\right) dr = 0.
 \end{equation}
Here we apply Lemma \ref{asymptotic linear} to deal with the term $\left|r^{\frac{d-3}{2}} u(r,t)\right|^2$:
\[
 \lim_{t\rightarrow +\infty} \int_{0}^\infty \left|r^{\frac{d-3}{2}} u(r,t)\right|^2 dr = \lim_{t \rightarrow +\infty} \frac{1}{c_d} \int_{\Rm^d} \frac{|u(x,t)|^2}{|x|^2} dx = 0.
\]
By Proposition \ref{energy flux} and Lemma \ref{asymptotic linear} we also have
\begin{align*}
 \lim_{R\rightarrow +\infty} \sup_{t>0} \left\{\int_{t+R}^\infty \left(\left|r^{\frac{d-1}{2}}u_r(r,t)\right|^2 + \left|r^{\frac{d-1}{2}} u_t(r,t)\right|^2 \right) dr\right\} & = 0;\\
 \lim_{\eta \rightarrow +\infty} \sup_{t>\eta} \left\{\int_{0}^{t-\eta} \left(\left|r^{\frac{d-1}{2}}u_r(r,t)\right|^2 + \left|r^{\frac{d-1}{2}} u_t(r,t)\right|^2 \right) dr\right\} & = 0.
\end{align*}
In addition, we may apply the change of variable $\eta' = t-r$ and obtain
\begin{align*}
 \lim_{R\rightarrow +\infty}  \sup_{t>0} \left\{\int_{t+R}^\infty \left|g_+(t-r)\right|^2 dr \right\}& = \lim_{R\rightarrow +\infty}  \int_{-\infty}^{-R} |g_+(\eta')|^2 d\eta' =0,\\
 \lim_{\eta\rightarrow +\infty}  \sup_{t>\eta} \left\{\int_{0}^{t-\eta} \left|g_+(t-r)\right|^2 dr \right\} & = \lim_{\eta\rightarrow +\infty}  \int_{\eta}^{\infty} |g_+(\eta')|^2 d\eta' =0.
\end{align*}
Combining the estimates of $u$ and $g_+$ given above, we have
\begin{align*}
 \lim_{R\rightarrow +\infty}\sup_{t >0} \left\{\int_{t+R}^{+\infty} \left(\left|r^{\frac{d-1}{2}}u_r(r,t) + g_+(t-r)\right|^2 + \left|r^{\frac{d-1}{2}} u_t(r,t) - g_+(t-r)\right|^2\right) dr\right\} & = 0;\\
 \lim_{\eta\rightarrow +\infty}\sup_{t >\eta} \left\{\int_{0}^{t-\eta} \left(\left|r^{\frac{d-1}{2}}u_r(r,t) + g_+(t-r)\right|^2 + \left|r^{\frac{d-1}{2}} u_t(r,t) - g_+(t-r)\right|^2\right) dr\right\} & = 0.
\end{align*}
Finally we combine these limits with \eqref{middle estimate linear} and obtain
\begin{proposition} \label{global L2 linear}
 Let $u$ be a radial finite-energy solution to the free wave equation. We have
 \begin{align*}
 \lim_{t \rightarrow + \infty} \int_{0}^\infty \left(\left|r^{\frac{d-1}{2}}u_r(r,t) + g_+(t-r)\right|^2 + \left|r^{\frac{d-1}{2}} u_t(r,t) - g_+(t-r)\right|^2\right) dr  = 0.
 \end{align*}
\end{proposition}

\section{Radial Linear Solutions} \label{sec: radiation}
We first give a lemma ($v_+, \tilde{v}_+$ are defined as in Definition \ref{def of wv})
\begin{lemma} \label{identity of radiation field}
 Assume that $u, \tilde{u} \in C(\Rm; \dot{H}_{rad}^1 \times L_{rad}^2)$ satisfy
 \begin{equation}
 \lim_{t \rightarrow + \infty} \left\|(u(\cdot,t), u_t(\cdot,t))-(\tilde{u}(\cdot,t), \tilde{u}_t(\cdot,t))\right\|_{\dot{H}^1\times L^2(\Rm^d)} = 0. \label{scattering condition}
\end{equation}
If the functions $v_+(t-\eta,t), \tilde{v}_+(t-\eta,t)$ converge in $L_{loc}^2(\Rm)$ to $2g_+(\eta)$ and $2\tilde{g}_+(\eta)$, respectively, when $t\rightarrow +\infty$,  then we must have $g_+ = \tilde{g}_+$.
\end{lemma}
\begin{proof}
 Given any $\eta_1<\eta_2$, we have
\begin{align*}
 4\int_{\eta_1}^{\eta_2} |g_+(\eta)-\tilde{g}_+(\eta)|^2 d\eta & = \lim_{t\rightarrow +\infty} \int_{\eta_1}^{\eta_2} |v_+(t-\eta,t)-\tilde{v}_+(t-\eta,t)|^2 d\eta \\
 & \leq \lim_{t\rightarrow +\infty} \int_{0}^{\infty} |v_+(r,t)-\tilde{v}_+(r,t)|^2 dr \\
 & \lesssim_d \lim_{t\rightarrow +\infty} \int_{\Rm^d} \left(|\nabla u-\nabla \tilde{u}|^2 + |u_t-\tilde{u}_t|^2\right) dx = 0.
\end{align*}
We apply Lemma \ref{relation L2 uv} in the argument above. It immediately follows that $g_+ = \tilde{g}_+$.
\end{proof}
\paragraph{Necessary condition of scattering} Now let us assume that a radial solution $u$ to (CP1) with a finite energy scatters in the positive time direction. Namely there exists a free wave $\tilde{u}= \mathbf{S}_L(\tilde{u}_0,\tilde{u}_1)$ so that \eqref{scattering condition} holds. Let $g_+, \tilde{g}_+$ be corresponding functions defined in Proposition \ref{introduction of g}. By Lemma \ref{identity of radiation field}, we must have $g_+ = \tilde{g}_+ = \mathbf{T}_+(\tilde{u}_0,\tilde{u}_1)$. As a result, we may expect a radial solution $u$ to scatter in the positive time direction only when the corresponding $g_+$ is contained in the image of the transformation $\mathbf{T}_+$ introduced in Definition \ref{def of T}. The majority of this section is devoted to the proof of the following proposition.

\begin{proposition} \label{homomorphism}
 The operator $\mathbf{T}_+$ introduced in Definition \ref{def of T} is a one-to-one isometry (up to a scalar multiplication) from $\dot{H}_{rad}^1 \times L_{rad}^2 (\Rm^d)$ to $L^2(\Rm)$. More precisely we have
 \begin{itemize}
  \item[(a)] $\left\|\mathbf{T}_+(u_0,u_1)\right\|_{L^2} = (2c_d)^{-1/2} \|(u_0,u_1)\|_{\dot{H}^1 \times L^2}$;
  \item[(b)] The linear operator $\mathbf{T}_+$ is a bijection. 
 \end{itemize}
\end{proposition}

\begin{remark}
Proposition \ref{homomorphism} (along with Proposition \ref{global L2 linear}) is actually the radial version of the following theorem known as ``radiation field'', the details and proof of which can be found in Duyckaerts et al. \cite{dkm3} and Friedlander \cite{radiation1, radiation2}. Although Proposition \ref{homomorphism} appears to be a direct corollary of Theorem \ref{radiation}, we still give our own proof in the radial case for completeness of our theory. 
\end{remark}

\begin{theorem}[Radiation filed] \label{radiation}
Assume that $d\geq 3$ and let $u$ be a solution to the free wave equation $\partial_t^2 u - \Delta u = 0$ with initial data $(u_0,u_1) \in \dot{H}^1 \times L^2(\Rm^d)$. Then 
\[
 \lim_{t\rightarrow +\infty} \int_{\Rm^d} \left(|\nabla u(x,t)|^2 - |u_r(x,t)|^2 + \frac{|u(x,t)|^2}{|x|^2}\right) dx = 0
\]
 and there exists a function $G_+ \in L^2(\Rm \times \mathbb{S}^{d-1})$ so that
\begin{align*}
 \lim_{t\rightarrow +\infty} \int_0^\infty \int_{\mathbb{S}^{d-1}} \left|r^{\frac{d-1}{2}} \partial_t u(r\theta, t) - G_+(r-t, \theta)\right|^2 d\theta dr &= 0;\\
 \lim_{t\rightarrow +\infty} \int_0^\infty \int_{\mathbb{S}^{d-1}} \left|r^{\frac{d-1}{2}} \partial_r u(r\theta, t) + G_+(r-t, \theta)\right|^2 d\theta dr & = 0.
\end{align*}
In addition, the map $(u_0,u_1) \rightarrow \sqrt{2} G_+$ is a bijective isometry form $\dot{H}^2 \times L^2(\Rm^d)$ to $L^2 (\Rm \times \mathbb{D}^{d-1})$.
\end{theorem}

\paragraph{Proof of isometry} This is a direct consequence of Proposition \ref{global L2 linear} ($u= \mathbf{S}_L(u_0,u_1)$)
\begin{align*}
 \|(u_0,u_1)\|_{\dot{H}^1 \times L^2}^2 &  = c_d \int_{0}^\infty \left(\left|r^{\frac{d-1}{2}}u_r(r,t)\right|^2 + \left|r^{\frac{d-1}{2}} u_t(r,t)\right|^2\right) dr  = \lim_{t \rightarrow +\infty} 2c_d \int_{0}^\infty |g_+(t-r)|^2dr \\
 & = \lim_{t \rightarrow +\infty} 2c_d \int_{-\infty}^t |g_+(\eta)|^2dr = 2c_d \|g_+\|_{L^{2}(\Rm)}^2.
\end{align*}

\paragraph{Proof of bijection}

Since the linear operator $\mathbf{T}_+$ preserves the norm up to a constant, we know that this must be one-to-one. It suffices to show that the image of this operator is dense in $L^2(\Rm)$. In fact we will show that the image contains all smooth and compactly supported functions $g \in C_0^\infty (\Rm)$. The argument consists of two major steps
\begin{itemize}
 \item Given any $g \in C_0^\infty (\Rm)$, we construct a function $\tilde{u}$ that comes with the desired asymptotic behaviour but solves the free wave equation only approximately. 
 \item We then modify $\tilde{u}$ slightly to obtain a solution $u$ that exactly solves the free wave equation and possesses the same asymptotic behaviour. 
\end{itemize}

\paragraph{Construction of $\tilde{u}$} Assume that $g$ is smooth and supported in $[-R,R]$. We define 
\[
 \tilde{u}(x,t) = - |x|^{-\frac{d-1}{2}} \int_{t-|x|}^{t+|x|} g(\eta) d\eta. 
\]
This function is smooth for $t > R$ because for these $t$ we always have $\tilde{u}(x,t) = 0$ in a neighbourhood of $x=0$, the only place where the smoothness might break down. The behaviour of $\tilde{u}$ when $|x|>t+ R$ can also be found by a simple calculation.
\begin{align*}
 &\tilde{u}(x,t) = C_1 |x|^{-\frac{d-1}{2}}, & &\nabla \tilde{u}(x,t) = C_2 |x|^{-\frac{d+3}{2}} x, & &\tilde{u}_t(x,t) = 0,& &|x|>t+R, t>R.&
\end{align*}
Thus we have $(\tilde{u}(\cdot,t), \tilde{u}_t(\cdot,t)) \in C((R,\infty); \dot{H}^1\times L^2)$. We then calculate $\tilde{w}, \tilde{v}_+$ accordingly  
\begin{align*}
 &\tilde{w}(r,t) = - \int_{t-r}^{t+r} g(\eta) d\eta,& &\tilde{v}_+(r,t) = 2g(t-r).& 
\end{align*}
It is clear that $\tilde{w}$ satisfies $(\partial_t^2 - \partial_r^2)\tilde{w} = 0$. Thus by identity \eqref{transformation uw} we have $(\partial_t^2 - \Delta) \tilde{u} = \lambda_d r^{-2} \tilde{u}$. As a result we have the following estimate for $t>R$
\[
  \left|(\partial_t^2 - \Delta) \tilde{u}(x,t)\right| \leq \left\{\begin{array}{ll} 0, & \hbox{if}\; |x|<t-R; \\
  \lambda_d |x|^{-\frac{d+3}{2}} \|g\|_{L^1}, & \hbox{if}\; |x|\geq t-R. \end{array}\right.
\]
This immediately gives us 
\begin{align*}
 \|(\partial_t^2 - \Delta) \tilde{u}\|_{L^1 L^2([2R,\infty)\times \Rm^d)} & \lesssim_d \|g\|_{L^1} \int_{2R}^\infty \left(\int_{|x|\geq t-R} |x|^{-(d+3)} dx\right)^{1/2} dt\\
 &  \lesssim_d \|g\|_{L^1} \int_{2R}^\infty (t-R)^{-3/2} dt \lesssim_1 \|g\|_{L^1} R^{-1/2}.
\end{align*}
Now we have collected sufficient information about our approximation solution $\tilde{u}$. The key tool to find a free wave $u$ with a similar asymptotic behaviour is the following lemma.

\begin{lemma} \label{u modification}
 Let $\tilde{u}$ be a solution to the wave equation 
 \[
   \partial_t^2 \tilde{u} (x,t) - \Delta \tilde{u}(x,t) = \tilde{F}(x,t), \quad (x,t) \in \Rm^d \times [T,\infty)
 \]
 with initial data $(\tilde{u}(\cdot,T), \tilde{u}_t(\cdot,T)) \in \dot{H}^1\times L^2$ and $\tilde{F} \in L^1 L^2([T,+\infty)\times \Rm^d)$. Then there exists a free wave $u$ so that 
 \[
  \lim_{t\rightarrow +\infty} \left\| (\tilde{u}(\cdot,t), \tilde{u}_t(\cdot,t)) - (u(\cdot,t), u_t(\cdot,t))\right\|_{\dot{H}^1 \times L^2} = 0.
 \]
 If $\tilde{u}$ is a radial solution, then $u$ is also radial.
\end{lemma}
\begin{proof}
First of all, we recall the fact that the linear wave propagation operator $\mathbf{S}_L (t)$ is unitary, apply the Strichartz estimates and obtain 
\begin{align*}
 & \limsup_{t_1,t_2\rightarrow +\infty} \left\|\mathbf{S}_L (-t_1)\begin{pmatrix}\tilde{u}(\cdot,t_1)\\ \tilde{u}_t(\cdot,t_1)\end{pmatrix} - \mathbf{S}_L (-t_2) \begin{pmatrix}\tilde{u}(\cdot,t_2)\\  \tilde{u}_t(\cdot,t_2)\end{pmatrix}\right\|_{\dot{H}^{1} \times L^2 (\Rm^d)}\\
 = &\limsup_{t_1,t_2\rightarrow +\infty} \left\|\mathbf{S}_L (t_2-t_1)\begin{pmatrix}\tilde{u}(\cdot,t_1)\\ \tilde{u}_t(\cdot,t_1)\end{pmatrix} - \begin{pmatrix}\tilde{u}(\cdot,t_2)\\ \tilde{u}_t(\cdot,t_2)\end{pmatrix}\right\|_{\dot{H}^{1} \times L^2 (\Rm^d)}\\
 \lesssim & \limsup_{t_1,t_2\rightarrow +\infty} \|F\|_{L^{1} L^{2}([t_1,t_2] \times \Rm^d)} = 0.
\end{align*}
Because the space $\dot{H}^1 \times L^2$ is complete, there exists $(u_0,u_1) \in \dot{H}^1 \times L^2$, so that
\begin{align*}
 \lim_{t\rightarrow +\infty} & \left\|\mathbf{S}_L (-t)\begin{pmatrix}\tilde{u}(\cdot,t)\\ \tilde{u}_t(\cdot,t)\end{pmatrix} -  \begin{pmatrix} u_0\\  u_1\end{pmatrix}\right\|_{\dot{H}^{1} \times L^2 (\Rm^d)} = 0 \\
 & \Rightarrow \lim_{t\rightarrow +\infty} \left\|\begin{pmatrix}\tilde{u}(\cdot,t)\\ \tilde{u}_t(\cdot,t)\end{pmatrix} -  \mathbf{S}_L (t) \begin{pmatrix} u_0\\  u_1\end{pmatrix}\right\|_{\dot{H}^{1} \times L^2 (\Rm^d)} = 0.
\end{align*}
Thus $u = \mathbf{S}_L (u_0,u_1)$ is the solution we are looking for. Finally if $\tilde{u}$ is radial, then $(u_0,u_1)$ must be radial as well, 
since $\dot{H}_{rad}^1 \times L_{rad}^2$ is a closed subspace of $\dot{H}^1 \times L^2$.  
\end{proof}

\paragraph{Completion of the proof} An application of Lemma \ref{u modification} on the approximation solution $\tilde{u}$ we constructed above gives a free wave $u = \mathbf{S}_L(u_0,u_1)$. We then define $w$, $v_+$, $g_+$ accordingly. Since $\tilde{v}_+(t-\eta,t)\equiv 2g(\eta)$ holds for all $t>\max\{R,\eta\}$, we may apply Lemma \ref{identity of radiation field} again to conclude $\mathbf{T}_+(u_0,u_1) = g_+ = g$.

\section{Global behaviour of solutions}

In this section we prove two main theorems. Assume that $u$ is a solution to (CP1) with a finite energy $E$. Let $g_+$ be the function defined in Proposition \ref{introduction of g}. By Proposition \ref{homomorphism} there exists a free wave $\tilde{u} = \mathbf{S}_L (\tilde{u}_0,\tilde{u}_1)$ so that $\mathbf{T}_+ (\tilde{u}_0,\tilde{u}_1)  = g_+$. Throughout this section we still use the same notations $v_\pm, g_+, \tilde{v}_\pm, \tilde{g}_+$ as defined in Section \ref{sec: reduction to 1D}.

\subsection{Scattering outside a light cone}
 In this section we prove part (a) of Theorem \ref{main 1}. First of all, we may compare the energy $\tilde{E}$ with $E$
\[
 \tilde{E}/c_d = \|g_+\|_{L^2}^2 \leq E/c_d \quad \Rightarrow \quad \tilde{E} \leq E. 
\]
thanks to Proposition \ref{homomorphism} and Proposition \ref{introduction of g}. We still need to show 
\[
 \lim_{t\rightarrow +\infty} \int_{|x|>t-\eta} \left(|\nabla u-\nabla \tilde{u}|^2 + |u_t - \tilde{u}_t|^2 \right) dx = 0.
\]
for any constant $\eta \in \Rm$. We start by splitting the integral above into two parts ($R > \max\{0,-\eta\}$)
 \begin{align*}
  \int_{|x|>t-\eta} \left(|\nabla u-\nabla \tilde{u}|^2 + |u_t - \tilde{u}_t|^2 \right) dx = & \int_{t-\eta<|x|<t+R} \left(|\nabla u-\nabla \tilde{u}|^2 + |u_t - \tilde{u}_t|^2 \right) dx \\
  & \qquad + \int_{|x|>t+R} \left(|\nabla u-\nabla \tilde{u}|^2 + |u_t - \tilde{u}_t|^2 \right) dx.
 \end{align*}
By Proposition \ref{energy flux}, the second term converges to zero uniformly for $t\in [0,\infty)$ as $R\rightarrow 0$, namely
\[
 \lim_{R\rightarrow +\infty} \left\{\sup_{t\geq 0} \int_{|x|>t+R} \left(|\nabla u-\nabla \tilde{u}|^2 + |u_t - \tilde{u}_t|^2 \right) dx\right\} = 0.
\]
Thus it suffices to prove the following limit for all fixed $R>\max\{0,-\eta\}$. 
\begin{equation}
  \lim_{t \rightarrow +\infty} \int_{t-\eta<|x|<t+R} \left(|\nabla u-\nabla \tilde{u}|^2 + |u_t - \tilde{u}_t|^2 \right) dx = 0. \label{to prove part a}
\end{equation}
This immediately follows 
\begin{lemma} \label{convergence u tilde}
 Let $u, \tilde{u}$ be defined as above. Given constants $c,R>0$ and $\gamma \in [0,2\beta(d,p))$, we have
 \begin{align*}
  \lim_{t \rightarrow +\infty} \int_{t-c\cdot t^\gamma<|x|<t+R} \left(|\nabla u-\nabla \tilde{u}|^2 + |u_t - \tilde{u}_t|^2 \right) dx & = 0;\\
  \limsup_{t \rightarrow +\infty} \int_{t-c\cdot t^{2\beta(d,p)}<|x|<t+R} \left(|\nabla u-\nabla \tilde{u}|^2 + |u_t - \tilde{u}_t|^2 \right) dx & \lesssim_{d,E} c;
 \end{align*}
\end{lemma}
\begin{proof}
 For any $\gamma \in [0,2\beta(d,p)]$, we may conduct a simple calculation
 \begin{align*}
  & \int_{t-c\cdot t^\gamma<|x|<t+R} \left(|\nabla u-\nabla \tilde{u}|^2 + |u_t - \tilde{u}_t|^2 \right) dx \\
  = & c_d \int_{t-c\cdot t^\gamma}^{t+R} \left(\left|r^{\frac{d-1}{2}}u_r- r^{\frac{d-1}{2}} \tilde{u}_r\right|^2 + \left|r^{\frac{d-1}{2}}u_t - r^{\frac{d-1}{2}}\tilde{u}_t\right|^2 \right) dr \\
  \leq & 2c_d \int_{t-c\cdot t^\gamma}^{t+R} \left(\left|r^{\frac{d-1}{2}}u_r+ g_+\right|^2 + \left|r^{\frac{d-1}{2}}u_t - g_+\right|^2 + \left|r^{\frac{d-1}{2}}\tilde{u}_r+ g_+\right|^2 + \left|r^{\frac{d-1}{2}}\tilde{u}_t - g_+\right|^2\right) dr 
 \end{align*}
We then evaluate the (upper) limits of the integrals in the last line above by Proposition \ref{middle convergence defocusing} and Proposition \ref{global L2 linear} to finish the proof.
\end{proof}

\subsection{Equivalent condition of scattering} 

In this subsection we prove part (b) of theorem, i.e. the solution $u$ scatters if and only if $\tilde{E}= E$. 

\paragraph{Scattering implies $\tilde{E}=E$} Let us assume 
\[
 \lim_{t\rightarrow +\infty} \left\|(u(\cdot,t), u_t(\cdot,t))-(\tilde{u}(\cdot,t)-\tilde{u}_t(\cdot,t))\right\|_{\dot{H}^1 \times L^2} = 0.
\]
This means
\[
 \lim_{t \rightarrow +\infty} \frac{1}{2} \int_{\Rm^d} \left(|\nabla u|^2 + |u_t|^2\right) = \tilde{E}\quad \Rightarrow \quad \lim_{t \rightarrow +\infty} \frac{1}{p+1} \int_{\Rm^d} |u(x,t)|^{p+1} dx = E-\tilde{E}.
\]
According to corollary \ref{limit of potential}, we also have 
\[
 \liminf_{t\rightarrow +\infty} \int_{\Rm^d} |u(x,t)|^{p+1} dx = 0.
\]
Thus we must have $\tilde{E}=E$.

\paragraph{$\tilde{E}=E$ implies scattering} Given any small constant $\eps>0$, by Lemma \ref{asymptotic linear} we can always find a constant $\eta \in \Rm^+$, so that 
\begin{equation}
 \sup_{t> \eta} \int_{|x|<t-\eta} \left(\frac{1}{2}|\nabla \tilde{u}(x,t)|^2 + \frac{1}{2}|\tilde{u}_t(x,t)|\right) dx < \eps. \label{small core}
\end{equation}
We recall the conclusion of part (a)
\begin{equation}
  \lim_{t\rightarrow +\infty} \int_{|x|>t-\eta} \left(|\nabla u(x,t)-\nabla \tilde{u}(x,t)|^2 + |u_t(x,t) - \tilde{u}_t(x,t)|^2 \right) dx = 0. \label{exterior scattering}
\end{equation}
We combine \eqref{small core}, \eqref{exterior scattering} and the energy conservation law of free wave equation to obtain
\begin{align*}
 \liminf_{t\rightarrow +\infty} &\int_{|x|>t-\eta} \left(\frac{1}{2}|\nabla u(x,t)|^2 + \frac{1}{2}|u_t(x,t)|\right) dx \\
 &= \liminf_{t\rightarrow +\infty} \int_{|x|>t-\eta} \left(\frac{1}{2}|\nabla \tilde{u}(x,t)|^2 + \frac{1}{2}|\tilde{u}_t(x,t)|\right) dx \geq \tilde{E}-\eps = E-\eps.
\end{align*}
The energy conservation law of defocusing equation then gives
\[
 \limsup_{t\rightarrow +\infty} \int_{|x|<t-\eta} \left(\frac{1}{2}|\nabla u(x,t)|^2 + \frac{1}{2}|u_t(x,t)|\right) dx \leq \eps. 
\]
Finally we combine this upper limit with \eqref{small core} and \eqref{exterior scattering} to conclude
\begin{align*}
 \limsup_{t\rightarrow +\infty} & \int_{|x|<t-\eta} \left(\frac{1}{2}|\nabla u(x,t)-\nabla \tilde{u}(x,t)|^2 + \frac{1}{2}|u_t(x,t)-\tilde{u}_t(x,t)|\right) dx \leq 4\eps;\\
 & \Rightarrow \limsup_{t\rightarrow +\infty} \int_{\Rm^d} \left(|\nabla u(x,t)-\nabla \tilde{u}(x,t)|^2 + |u_t(x,t)-\tilde{u}_t(x,t)|\right) dx \leq 8\eps.
\end{align*}
This finishes the proof because we may choose arbitrarily small constant $\eps$. 

\subsection{Scattering by energy decay}

In this subsection we prove Theorem \ref{main 2}. We start by explaining the basic idea. Our goal is to show 
\[
 \lim_{t\rightarrow +\infty} \int_{\Rm^d} \left(|\nabla u-\nabla \tilde{u}|^2 + |u_t-\tilde{u}_t|^2\right) dx = 0.
\]
We split the whole space $\Rm^d$ into three regions: $\Sigma_1(t) = \{x\in \Rm^d: |x|<t-ct^{2\beta(d,p)}\}$, $\Sigma_2(t)=\{x\in \Rm^d: t-ct^{2\beta(d,p)}<|x|<t+R\}$ and $\Sigma_3(t)=\{x\in \Rm^d: |x|>t+R\}$. Here $\beta(d,p)$ is defined in Proposition \ref{upper bound of variation v}; $c$ and $R$ are arbitrary positive constants. We then write the integral above as a sum of integrals over these three regions
\[
 \int_{\Rm^d}  \left(|\nabla u(x,t)-\nabla \tilde{u}(x,t)|^2 + |u_t(x,t)-\tilde{u}_t(x,t)|^2\right) dx = I_1(t) + I_2(t) + I_3(t),
\]
with 
\[
 I_j(t) = \int_{\Sigma_j(t)} \left(|\nabla u(x,t)-\nabla \tilde{u}(x,t)|^2 + |u_t(x,t)-\tilde{u}_t(x,t)|^2\right) dx.
\]
The scattering of solution outside the forward light cone $|x|=t+R$ has been proved, namely 
\[
  \lim_{t\rightarrow +\infty} I_3(t) = 0.
\]
In addition, we may apply Proposition \ref{convergence u tilde} and obtain 
\[
 \limsup_{t\rightarrow +\infty} I_2(t) \lesssim_{d,E} c.
\]
We still need to consider the limit of $I_1(t)$. This is clear that 
\[
 I_1(t) \lesssim_1 \int_{\Sigma_1(t)} \left(|\nabla u(x,t)|^2 + |u_t(x,t)|^2\right) dx + \int_{\Sigma_1(t)} \left(|\nabla \tilde{u}(x,t)|^2 + |\tilde{u}_t(x,t)|^2\right) dx.
\]
The latter term converges to zero as $t\rightarrow +\infty$, according to the asymptotic behaviour of free waves given in Lemma \ref{asymptotic linear}. The former term can be dealt with by the following proposition. 

\begin{proposition} \label{center decay}
Assume that $\kappa \in (0,1)$ is a constant. Let $u$ be a solution to (CP1) with initial data $(u_0,u_1)$ so that
 \[
  E_\kappa(u_0,u_1)  \doteq \int_{\Rm^d} (|x|^\kappa+1)\left(\frac{1}{2}|\nabla u_0(x)|^2 + \frac{1}{2}|u_1(x)|^2 + \frac{1}{p+1}|u_0(x)|^{p+1}\right) dx < +\infty. 
 \]
 Then we have the following limit regarding local energy for any constant $c>0$. 
 \[
  \lim_{t\rightarrow +\infty} \int_{|x|<t-c\cdot t^{1-\kappa}} \left(\frac{1}{2}|\nabla u(x,t)|^2 + \frac{1}{2}|u_t(x,t)|^2 + \frac{1}{p+1}|u(x,t)|^{p+1}\right) dx = 0.
 \]
\end{proposition}

\noindent We postpone the proof of Proposition \ref{center decay} until the final part of this section and first complete the proof of Theorem \ref{main 2}. Let us recall our assumption $\kappa \geq \kappa_0(d,p) = 1 - 2\beta(d,p)$. Thus we have $1-\kappa \leq 2\beta(d,p)$. We may apply Proposition \ref{center decay} and obtain\footnote{Without loss of generality we may also assume $\kappa<1$. Otherwise we may substitute $\kappa$ by an arbitrary $\kappa' \in [k_0(d,p),1)$ because $E_{\kappa'}(u_0,u_1)\lesssim_1 E_\kappa(u_0,u_1)<+\infty$.}
\[
 \limsup_{t\rightarrow +\infty} \int_{\Sigma_1(t)} \left(|\nabla u|^2 + |u_t|^2\right) dx \leq \lim_{t\rightarrow +\infty} \int_{|x|<t-c\cdot t^{1-\kappa}} \left(|\nabla u|^2 + |u_t|^2\right) dx = 0.
\]
Thus $I_1(t) \rightarrow 0$. We collect the (upper) limits of all three terms $I_1(t)$, $I_2(t)$, $I_3(t)$ and put them together
\[
 \limsup_{t\rightarrow +\infty} \int_{\Rm^d}  \left(|\nabla u(x,t)-\nabla \tilde{u}(x,t)|^2 + |u_t(x,t)-\tilde{u}_t(x,t)|^2\right) dx \lesssim_{d,E} c.
\]
Now we are able to conclude the proof of Theorem \ref{main 2} by the arbitrariness of $c$. We conclude this section by giving the proof of Proposition \ref{center decay}.

\begin{proof}[Proof of Proposition \ref{center decay}]
For convenience let us use the notation of local energy introduced in Subsection \ref{sec: energy flux}. 
\[
 E(t;\Sigma) = \int_{\Sigma} \left(\frac{1}{2}|\nabla u(x,t)|^2 + \frac{1}{2}|u_t(x,t)|^2 + \frac{1}{p+1}|u(x,t)|^{p+1}\right) dx.
\]
We fix a large time $t$ and use the finite movement speed of energy as given in Proposition \ref{energy flux} to obtain 
\begin{align*}
  E(t; B(0,|x|<t-c\cdot t^{1-\kappa})) \leq E(t';B(0,t'-c\cdot t^{1-\kappa})), \; \forall t'\geq t.
\end{align*}
Thus we have
\begin{align}
  E(t; B(0,|x|<t-c\cdot t^{1-\kappa})) \leq & \frac{1}{c t^{1-\kappa}}\int_t^{t+c\cdot t^{1-\kappa}} E(t';B(0,t'-c\cdot t^{1-\kappa})) dt'\nonumber\\
  \leq & \frac{1}{c t^{1-\kappa}}\int_t^{\infty} E(t';B(0,t)) dt' \nonumber\\
  \leq & \frac{1}{c t^{1-\kappa}}\int_{-t}^{t} E(t'; \{x\in \Rm^d: |x|>t\}) dt'.\label{upper bound integral small t}
\end{align}
We apply Proposition \ref{energy distribution} in the last step above. Next we use the finite movement speed of energy again to give an upper bound of the local energy involved in \eqref{upper bound integral small t} 
\begin{align*}
 E(t'; \{x\in \Rm^d: |x|>t\}) & \leq \int_{|x|>t-|t'|} \left(\frac{|\nabla u_0(x)|^2}{2} + \frac{|u_1(x)|^2}{2} + \frac{|u_0(x)|^{p+1}}{p+1}\right) dx \\
 & \leq (t-|t'|)^{-\kappa} E_\kappa (u_0,u_1;t-|t'|). 
\end{align*}
Here $E_\kappa(u_0,u_1;r)$ is a decreasing function of $r$ defined by
\[
 E_\kappa(u_0,u_1;r) = \int_{|x|>r} (|x|^\kappa+1)\left(\frac{1}{2}|\nabla u_0(x)|^2 + \frac{1}{2}|u_1(x)|^2 + \frac{1}{p+1}|u_0(x)|^{p+1}\right) dx.
\]
It converges to zero as $r\rightarrow +\infty$. In addition, it is clear that $E(t'; \{x\in \Rm^d: |x|>t\}) \leq E$ always hold for all $t'$. As a result, we may find an upper bound of the integral in \eqref{upper bound integral small t}. 
\begin{align*}
 \int_{-t}^{t} E(t'; \{x\in \Rm^d: |x|>t\}) dt' & \leq \int_{-t+t^{(1-\kappa)/2}}^{t-t^{(1-\kappa)/2}} E(t'; \{x\in \Rm^d: |x|>t\}) dt' + 2t^{(1-\kappa)/2} E\\
 & \leq \int_{-t+t^{(1-\kappa)/2}}^{t-t^{(1-\kappa)/2}} (t-|t'|)^{-\kappa} E_\kappa (u_0,u_1;t-|t'|) dt' + 2t^{(1-\kappa)/2} E\\
 & \leq \int_{-t+t^{(1-\kappa)/2}}^{t-t^{(1-\kappa)/2}} (t-|t'|)^{-\kappa} E_\kappa (u_0,u_1;t^{(1-\kappa)/2}) dt' + 2t^{(1-\kappa)/2} E\\
 & \leq \frac{2t^{1-\kappa}}{1-\kappa} E_\kappa (u_0,u_1;t^{(1-\kappa)/2}) + 2t^{(1-\kappa)/2} E.
\end{align*}
Finally we plug this upper bound in \eqref{upper bound integral small t}, let $t\rightarrow +\infty$ and finish the proof
\[
 E(t; B(0,|x|<t-c\cdot t^{1-\kappa})) \leq \frac{2}{(1-\kappa)c} E_\kappa (u_0,u_1;t^{(1-\kappa)/2}) + \frac{2E}{c} t^{-(1-\kappa)/2}\rightarrow 0.
\]
\end{proof}


\begin{thebibliography}{99}
 \bibitem{claim1} B. Dodson. {``Global well-posedness and scattering for the radial, defocusing, cubic nonlinear wave equation.''} \textit{arXiv Preprint} 1809.08284.
 \bibitem{cubic3dwave} B. Dodson and A. Lawrie. {``Scattering for the radial 3d cubic wave equation.''} \textit{Analysis and PDE}, 8(2015): 467-497.
 \bibitem{nonradial3p5} B. Dodson, A. Lawrie, D. Mendelson, J. Murphy {``Scattering for defocusing energy subcritical nonlinear wave equations''}, \textit{arXiv Preprint} 1810.03182.
 \bibitem{dkm2} T. Duyckaerts, C.E. Kenig, and F. Merle. {``Scattering for radial, bounded solutions of focusing supercritical wave equations.''} \textit{International Mathematics Research Notices} 2014:  224-258.
 \bibitem{dkm3} T. Duyckaerts, C.E. Kenig, and F. Merle. {``Scattering profile for global solutions of the energy-critical wave equation.''} \textit{Journal of European Mathematical Society} 21 (2019): 2117-2162.
 \bibitem{pdeevans} L. C. Evans {``Partial Differential Equations, Second Edition.''} \textit{Graduate Studies in Mathematics} 19(2010), AMS, Providence.
\bibitem{radiation1} F. G. Friedlander. {``On the radiation field of pulse solutions of the wave equation.''}  \textit{Proceeding of the Royal Society Series A} 269 (1962): 53-65.
\bibitem{radiation2} F. G. Friedlander. {``Radiation fields and hyperbolic scattering theory.''} \textit{Mathematical Proceedings of Cambridge Philosophical  Society} 88(1980): 483-515.
 \bibitem{conformal2} J. Ginibre, and G. Velo. {``Conformal invariance and time decay for nonlinear wave equations.''} \textit{Annales de l'institut Henri Poincar\'{e} (A) Physique th\'{e}orique} 47(1987): 221-276.
 \bibitem{strichartz} J. Ginibre, and G. Velo. {``Generalized Strichartz inequality for the wave equation.''} \textit{Journal of Functional Analysis} 133(1995): 50-68.
 \bibitem{mg1} M. Grillakis. {``Regularity and asymptotic behaviour of the wave equation with critical nonlinearity.''} \textit{Annals of Mathematics} 132(1990): 485-509.
 \bibitem{mg2} M. Grillakis. {``Regularity for the wave equation with a critical nonlinearity.''} \textit{Communications on Pure and Applied Mathematics} 45(1992): 749-774.
 \bibitem{conformal} K. Hidano. {``Conformal conservation law, time decay and scattering for nonlinear wave equation''} \textit{Journal D'analysis Math\'{e}matique} 91(2003): 269-295.
\bibitem{loc1} L. Kapitanski. {``Weak and yet weaker solutions of semilinear wave equations''} \textit{Communications in Partial Differential Equations} 19(1994): 1629-1676.
 \bibitem{endpointStrichartz} M. Keel, and T. Tao. {``Endpoint Strichartz estimates''} \textit{American Journal of Mathematics} 120 (1998): 955-980.
\bibitem{kenig} C. E. Kenig, and F. Merle. {``Global Well-posedness, scattering and blow-up for the energy critical focusing non-linear wave equation.''} \textit{Acta Mathematica} 201(2008): 147-212.
\bibitem{kenig1} C. E. Kenig, and F. Merle. {``Global well-posedness, scattering and blow-up for the energy critical, focusing, non-linear Schr\"{o}dinger equation in the radial case.''} \textit{Inventiones Mathematicae} 166(2006): 645-675.
 \bibitem{km} C. E. Kenig, and F. Merle. {``Nondispersive radial solutions to energy supercritical non-linear wave equations, with applications.''} \textit{American Journal of Mathematics} 133, No 4(2011): 1029-1065.
 \bibitem{kv2} R. Killip, and M. Visan. {``The defocusing energy-supercritical nonlinear wave equation in three space dimensions''} \textit{Transactions of the American Mathematical Society}, 363(2011): 3893-3934.
 \bibitem{kv3} R. Killip, and M. Visan. {``The radial defocusing energy-supercritical nonlinear wave equation in all space dimensions''} \textit{Proceedings of the American Mathematical Society}, 139(2011): 1805-1817.
 \bibitem{ls} H. Lindblad, and C. Sogge. {``On existence and scattering with minimal regularity for semi-linear wave equations''} \textit{Journal of Functional Analysis} 130(1995): 357-426.
 \bibitem{benoit} B. Perthame, and L. Vega. {``Morrey-Campanato estimates for Helmholtz equations.''} \textit{Journal of Functional Analysis} 164(1999): 340-355.
 \bibitem{sub45} C. Rodriguez. {``Scattering for radial energy-subcritical wave equations in dimensions 4 and 5.''} \textit{Communications in Partial Differential Equations} 42(2017): 852-894. 
\bibitem{ss1} J. Shatah, and M. Struwe. {``Regularity results for nonlinear wave equations''} \textit{Annals of Mathematics} 138(1993): 503-518.
 \bibitem{ss2} J. Shatah, and M. Struwe. {``Well-posedness in the energy space for semilinear wave equations with critical growth''} \textit{International Mathematics Research Notices} 7(1994): 303-309.
 \bibitem{shen2} R. Shen. {``On the energy subcritical, nonlinear wave equation in $\Rm^3$ with radial data''}  \textit{Analysis and PDE} 6(2013): 1929-1987.
 \bibitem{shenenergy} R. Shen. {``Energy distribution of radial solutions to energy subcritical wave equation with an application on scattering theory''} \textit{arXiv Preprint} 1808.08656.
 \bibitem{shen3dnonradial} R. Shen {``Inward/outward Energy Theory of Non-radial Solutions to 3D Semi-linear Wave Equation''} \textit{arXiv Preprint} 1910.09805.
 \bibitem{shenhd} R. Shen {``Inward/outward Energy Theory of Wave Equation in Higher Dimensions''} \textit{arXiv Preprint} 1912.02428.
 \bibitem{yang1} S. Yang {``Global behaviors of defocusing semilimear wave equations''} \textit{arXiv Preprint} 1908.00606.
\end{thebibliography}
\end{document}